\documentclass[12pt]{amsart}

\usepackage{color,footnote}

\usepackage{amsmath,amssymb,setspace,nicefrac, yhmath, amscd,eucal}
\usepackage[active]{srcltx}
\usepackage[colorlinks, linkcolor=red, citecolor=blue, urlcolor=blue, hypertexnames=true]{hyperref}
\usepackage{amsrefs}

\allowdisplaybreaks

\newcommand{\R}{{\mathbb R}}

\newtheorem{thm}{Theorem}[section]
\newtheorem{cor}[thm]{Corollary}
\newtheorem{lem}[thm]{Lemma}

\theoremstyle{definition}

\begin{document}
\title[Cohomology groups invariant under continuous orbit equivalence]{Cohomology groups invariant under continuous orbit equivalence}
\author{Yongle Jiang}
\address{Institute of Mathematics, Polish Academy of Sciences, Warsaw, 00-656, Poland.}
\email{yjiang@impan.pl}
\date{\today}
\maketitle

\begin{abstract}
By the work of Brodzki-Niblo-Nowak-Wright and Monod, topological amenability of a continuous group action can be characterized using uniformly finite homology groups or bounded cohomology groups associated to this action. We show that (certain variations of) these groups are invariants for topologically free actions under continuous orbit equivalence.  
\end{abstract}

\section{Introduction}

We continue our study of continuous orbit equivalence introduced by Li \cite{li_0}. In this paper, we focus on certain (co)homology groups for a continuous group action. Let us first review the introduction of these (co)homology groups.

Let $G$ be a countable discrete group. There are two remarkable characterizations of amenablility of $G$. One, given by Johnson-Ringrose \cite{johnson}, says that $G$ is amenable if and only if the first bounded cohomology group with coefficients in $\ell^1_0(G)^{**}$ vanishes, i.e. $H_b^1(G, \ell^1_0(G)^{**})=0$, where $\ell^1_0(G)$ denotes the augmentation ideal, i.e. kernel of the summation map from $\ell^1(G)$ to $\mathbb{R}$. By contrast, Block and Weinberger \cite{bw} described amenability in terms of non-vanishing of the 0-dimensional uniformly finite homology of $G$, i.e. $H_0^{uf}(G, \mathbb{R})\neq 0$. For a short, unified proof of these two results, see \cite{bnw}. In particular, it was observed there that $H_*^{uf}(G, \mathbb{R})=H_*(G, \ell^{\infty}(G))$ if $G$ is finitely generated.

The notion of amenable actions of groups acting on topological spaces generalizes the concept of amenability and appears in many areas of mathematics. For example, a group acts amenably on a point if and only if it is amenable, while every hyperbolic group acts amenably on its Gromov boundary. For more on amenable actions, see \cite{AR, gk, hr, oz}.

Parallel with these characterizations, people also found two similar characterizations for the amenablity of group actions in the topological sense. To do this, the key step is to find appropriate coefficient modules associated to a continuous action $G\curvearrowright X$.

More precisely, Brodzki, Niblo, Nowak and Wright considered the standard module $W_0(G, X)$ and its submodule $N_0(G, X):=C(X, \ell^1_0(G))$. Note that $W_0(G, X)^*$ and $N_0(G, X)^{**}$ reduce to $\ell^{\infty}(G)$ and $\ell^1_0(G)^{**}$ respectively when $X$ is a point. Naturally, they considered the bounded cohomology groups with coefficients in $N_0(G, X)^{**}$ \cite{bnnw2} and also the uniformly finite homology of an action, $H^{uf}_n(G\curvearrowright X)$ as the group homology with coefficients in $W_0(G, X)^*$, i.e. $H^{uf}_n(G\curvearrowright X):=H_n(G, W_0(G, X)^*)$ \cite{bnnw}. They succeeded in characterizing amenability of actions using these (co)homology groups, generalizing the above results of Johson-Ringrose and Block-Weinberger for group case. A similar approach was also taken by Monod in \cite{monod}.

In \cite{bw} (see also \cite{li}), among other results, $H_*^{uf}(G, \mathbb{R})$ is shown to be an invariant for groups under coarse equivalence, i.e. quasi-isometry if the groups are finitely generated. Hence it is natural to ask whether the above (co)homology groups for actions are also invariants under some ``coarse equivalence" for actions. 

In this paper, we show this is indeed possible if we take ``coarse equivalence" to be ``continuous orbit equivalence", and actions are assumed to be topologically free.

Let us recall the definition of continuous orbit equivalence, for known results on this notion and its connection to geometric group theory, see \cite{cohen, cj1, cj2,j,li_0, li}.

Let $G\curvearrowright X$ and $H\curvearrowright Y$ be two actions by homeomorphisms, where $G$, $H$ are countable discrete groups and $X$, $Y$ are compact Hausdorff spaces. Following \cite{li_0}, we say the two actions are \emph{continuous orbit equivalent} (abbreviated as COE) if there are homeomorphisms $\phi: X\approx Y$, $\psi: Y\approx X$ and continuous maps $c: G\times X\to H$, $c': H\times Y\to G$ such that $\phi(gx)=c(g, x)\phi(x)$ and $\psi(hy)=c'(h, y)\psi(y)$ hold for all $g\in G$, $h\in H$, $x\in X$ and $y\in Y$. If these two actions are \emph{topologically free}, i.e. points with trivial stabilizers are dense, then both $c$ and $c'$ are cocycles \cite[Lemma 2.8]{li_0}. Recall that $c: G\times X\to H$ is a \emph{cocycle} if $c(g_1g_2, x)=c(g_1, g_2x)c(g_2, x)$ for all $g_1, g_2\in G$ and all $x\in X$.

Now, we can state our main theorems. Note that all acting groups are assumed to be countable discrete and spaces are assumed to be compact Hausdorff.

\begin{thm}\label{thm on uf homology}
Let $G\curvearrowright X$ and $H\curvearrowright Y$ be topologically free actions. If these two actions are COE, then $H^{uf}_0(G\curvearrowright X)\cong H^{uf}_0(H\curvearrowright Y)$ and $H_n(G, N_0(G, X)^*)\cong H_n(H, N_0(H, Y)^*)$ for all $n\geq 0$.
\end{thm}

\begin{thm}\label{thm on b cohomology}
Let $G\curvearrowright X$ and $H\curvearrowright Y$ be topologically free actions. If these two actions are COE, then $H^p_b(G,N_0(G, X)^{**})\cong H^p_b(H, N_0(H, Y)^{**})$ for all $p\geq 0$.
\end{thm}

The paper is organized as follows.

Besides the introduction, the paper contains five other sections. In Section \ref{section: preliminaries}, we review the definitions of group (co)homology and certain coefficient modules. In Section \ref{section: lemma}, we state several lemmas on basic properties of the orbit cocycles and observe that the coefficient modules are closed under restriction maps. These are the crucial ingredients for constructing bijective (co)chain maps later. The next two sections will take up the majority of our paper: 
Section \ref{section: proof of uf-homology case} proves Theorem \ref{thm on uf homology} and Section \ref{section: proof of b-cohomology case} proves Theorem \ref{thm on b cohomology}. The proof of these two theorems uses the same idea: we construct maps between (co)chain complexes directly and then use the whole sections to verify these maps are bijective (co)chain maps. Detailed proof for lower dimensional cases will be included to illustrate the main ideas. We conclude with several remarks in Section \ref{section: remarks} to discuss corollaries of the theorems and other related questions.

Following the convention in \cite{bnnw2, bnnw, monod}, all Banach spaces are assumed to be real.

\section{Preliminaries}\label{section: preliminaries}

\subsection{Group (co)homology}
We briefly recall the definition of group homology and (bounded) group cohomology using bar resolutions, see \cite[\S 4.3]{li}, \cite[Chapter III, \S 1]{brown} and \cite{monod_book}.

Let $G$ be a group and $V$ be a $\mathbb{Z}G$-module. Let $(C_*(V), \partial_*)$ be the chain complex $\dots\overset{\partial_3}{\rightarrow}C_2(V)\overset{\partial_2}{\rightarrow}C_1(V)\overset{\partial_1}{\rightarrow}C_0(V)$
with $C_0(V)=V$, $C_n(V)=C_f(G^n, V)$, where $C_f$ stands for maps with finite support, and
$\partial_n=\sum_{i=0}^n(-1)^i\partial_n^{(i)}$, where
\begin{align*}
\partial_n^{(0)}(f)(g_1,\dots, g_{n-1})&=\sum_{g_0\in G}g_0^{-1}f(g_0,g_1, \dots, g_{n-1}),\\
\partial_n^{(i)}(f)(g_1,\dots, g_{n-1})&=\sum_{\substack{g, \bar{g}\in G\\ g\bar{g}=g_i}}f(g_1, \dots, g_{i-1}, g, \bar{g}, g_{i+1}, g_{n-1})~\text{for} ~1\leq i\leq n-1,\\
\partial_n^{(n)}(f)(g_1,\dots, g_{n-1})&=\sum_{g_n\in G}f(g_1, \dots, g_{n-1}, g_n).
\end{align*}

Then we define $H_n(G, V):=Ker(\partial_n)/Im(\partial_{n+1})$.

Let $(C^*(V), \partial^*)$ be the cochain complex 
$C^0(V)\overset{\partial^0}{\rightarrow}C^1(V)\overset{\partial^1}{\rightarrow}C^2(V)\overset{\partial^2}{\rightarrow}\dots$ with $C^0(V)=V$, $C^n(V)=C(G^n, V)$ for all $n\geq 1$, and $\partial^n=\sum_{i=0}^{n+1}(-1)^i\partial_{(i)}^n$, where
\begin{align*}
\partial_{(0)}^n(f)(g_0,\dots,g_n)&=g_0f(g_1,\dots,g_n),\\
\partial_{(i)}^n(f)(g_0,\dots,g_n)&=f(g_0,\cdots,g_{i-1}g_i,\dots,g_n)~\text{for}~1\leq i\leq n,\\
\partial_{(n+1)}^n(f)(g_0,\dots,g_n)&=f(g_0,\dots,g_{n-1}).
\end{align*}

Then we define $H^n(G, V):=Ker(\partial^n)/Im(\partial^{n-1})$. 

If $V$ is a Banach $G$-module, we may replace $C^*(V)$ by $C^*_b(V)$, where $C^n_b(V):=C_b(G^n, V)$ and $C_b$ stands for uniformly bounded maps, then the group cohomology we get is the bounded cohomology group, written as $H^n_b(G, V)$.

\subsection{Coefficient modules}
The module $V$ we are interested in are certain dual or double dual Banach spaces, i.e.  $N_0(G, X)^*$, $W_0(G, X)^*$ and $N_0(G, X)^{**}$ associated to a continuous action $G\curvearrowright X$. 
Let us recall their definition below \cite{bnnw, bnnw2}.

The space $C(X, \ell^1(G))$ of continuous $\ell^1(G)$-valued functions on $X$ is equipped with the sup-$\ell^1$ norm \[||\xi||=\sup_{x\in X}\sum_{g\in G}|\xi_x(g)|.\]

Here, for $\xi\in C(X, \ell^1(G))$, we write $\xi_x(g):=\xi(x)(g)$. 
We remind the reader that we are using a different notation from the one in \cite{bnnw, bnnw2}, where $\xi_g(x)$ was used.

In this notation, the Banach space $C(X, \ell^1(G))$ is equipped with a natural action of $G$, 
\[(g\cdot \xi)_x(h):=\xi_{g^{-1}x}(g^{-1}h),\]
for each $g, h\in G$ and $x\in X$. 

The summation map on $\ell^1(G)$ induces a continuous map $\sigma: C(X, \ell^1(G))\to C(X)$, where $C(X)$ is equipped with the $\ell^{\infty}$ norm. The space $N_0(G, X)$ is defined to be the pre-image $\sigma^{-1}(0)$ which we identify as $C(X, \ell^1_0(G))$. Then we define $W_0(G, X):=\sigma^{-1}(\mathbb{R})$, where $\mathbb{R}$ is regarded as constant functions on $X$.
Note that $W_0(G, X)=N_0(G, X)\oplus \mathbb{R}$, where $\mathbb{R}$ is identified as constant $\mathbb{R}\delta_e$-valued functions on $X$, here $e$ is the neutral element in $G$. Obviously, both $N_0(G, X)$ and $W_0(G, X)$ are invariant under the above $G$-action since $\sigma$ is $G$-equivariant.
Also recall that if $V$ is a Banach $G$-module, then the Banach dual space $V^*$ is also equipped with a $G$-module structure by requiring $(g\phi)(\xi)=\phi(g^{-1}\xi)$ for $g\in G$, $\phi\in V^*$ and $\xi\in V$.

\section{Properties of orbit cocycles and restriction invariant property} \label{section: lemma}
In this section, we state several lemmas proving basic properties of orbit cocycles and observe that the module $N_0(G, X)$ is restriction invariant. These give us the hint on how to define bijective (co)chain maps between two (co)chain complexes later.

From now on, we use the notations introduced when defining COE, i.e. $\phi: X\approx Y$ and $\psi: Y\approx X$ are homeomorphisms; $c: G\times X\to H$ and $c': H\times Y\to G$ are maps satisfying certain identities.
Concerning COE between topologically free actions, the following property will be used frequently.

\begin{lem}\cite[Lemma 2.10]{li_0} \label{lemma: xin's lemma}
Let $G\curvearrowright X$ and $H\curvearrowright Y$ be topologically free actions that are COE. Then $c'(c(g, x), \phi(x))=g$ for all $g\in G$, $x\in X$. Similarly, $c(c'(h, y),\psi(y))=h$ for all $h\in H$, $y\in Y$.
\end{lem}
Clearly, this implies that for all $x\in X$, $G\ni g\mapsto c(g, x)\in H$ is a bijection.

The starting point for our proof is the following lemma.

\begin{lem}\label{lemma: key lemma}
Let $G\curvearrowright X$ and $H\curvearrowright Y$ be topologically free actions which are COE. Then there exists a (linear) map $\pi$ from $C(Y, \ell^1(H))$ onto $C(X, \ell^1(G))$ which is an isometry, and which moreover has the following properties: $\pi(W_0(H, Y))=W_0(G, X)$, $\pi(N_0(H, Y))=N_0(G, X)$ and $\pi(\mathbb{R})=\mathbb{R}$, where $\mathbb{R}$ is identified as constant $\mathbb{R}\delta_e$-valued functions on $X$.
\end{lem}
\begin{proof}
We define $\pi: C(Y,\ell^1(H))\to C(X, \ell^1(G))$ by setting $\pi(\xi)=\xi'$, where \[\xi'_x(g):=\xi_{\phi(x)}(c(g, g^{-1}x)).\]

First, we check $\pi$ is a well-defined isometry.

Note that for all $x\in X$, $G\ni g\mapsto c(g^{-1}, x)^{-1}=c(g, g^{-1}x)\in H$ is a bijection. A calculation shows that $\sum_g\xi'_x(g)=\sum_h\xi_{\phi(x)}(h)$ for all $x\in X$ and hence  $||\xi||=||\xi'||$. So $\pi$ is an isometry.

Now we check $\xi'\in C(X, \ell^1(G))$. For any given $x\in X$ and $\epsilon>0$, we need to find $\mathcal{U}_x$, a neighborhood of $x$, such that if $x'\in \mathcal{U}_x$, then $\sum_{g\in G}|\xi'_x(g)-\xi'_{x'}(g)|<\epsilon$. 

Since $\xi\in C(Y, \ell^1(H))$, there exists a small neighborhood $\mathcal{V}_{\phi(x)}$ containing $\phi(x)$ such that if $y'\in\mathcal{V}_{\phi(x)}$, then $\sum_{h\in H}|\xi_{y'}(h)-\xi_{\phi(x)}(h)|<\epsilon/4$. 

Take a large finite set $F\subseteq H$ such that $\sum_{h\not\in F}|\xi_{\phi(x)}(h)|<\epsilon/4$. This implies that if $y'\in\mathcal{V}_{\phi(x)}$, then \[\sum_{h\not\in F}|\xi_{y'}(h)|\leq \sum_{h\not\in F}|\xi_{y'}(h)-\xi_{\phi(x)}(h)|+\sum_{h\not\in F}|\xi_{\phi(x)}(h)|\leq \epsilon/2.\] 

Now, we take $\mathcal{U}_x$ to be any sufficiently small neighborhood of $x$, which is contained in $ \psi(\mathcal{V}_{\phi(x)})$ and satisfies the property that for any $x'\in \mathcal{U}_x$, we have $c(g^{-1}, x)=c(g^{-1}, x')$ for all $g\in K$, where $K:=\{g\in G: ~ c(g^{-1}, x)^{-1}\in F\}$ is a finite set. 

Take any $x'\in \mathcal{U}_x$, note that $g\in K$ iff $c(g^{-1}, x)^{-1}\in F$ iff $c(g^{-1}, x')^{-1}\in F$. Moreover, $x'\in \mathcal{U}_x$ implies $\phi(x')\in \mathcal{V}_{\phi(x)}$.
Hence, 
\begin{align*}
\sum_{g\in G}|\xi'_x(g)-\xi'_{x'}(g)|&=\sum_{g\in K}|\xi'_x(g)-\xi'_{x'}(g)|+\sum_{g\not\in K}|\xi'_x(g)-\xi'_{x'}(g)|\\
&\leq \sum_{h\in F}|\xi_{\phi(x)}(h)-\xi_{\phi(x')}(h)|+\sum_{h\not\in F}|\xi_{\phi(x)}(h)|+\sum_{h\not\in F}|\xi_{\phi(x')}(h)|\\
&\leq  \epsilon/4+\epsilon/4+\epsilon/2=\epsilon.
\end{align*}

Second, it is easy to check $\pi(W_0(H, Y))\subseteq W_0(G, X), \pi(N_0(H, Y))\subseteq N_0(G, X)$ and $\pi(\R)\subseteq \R$ hold. To see $\subseteq$ is really $=$, we need to find the inverse of $\pi$. 

By symmetry, we can define another isometry $L: C(X, \ell^1(G))\to C(Y, \ell^1(H))$ by setting $L(\eta)=\eta'$, where \[\eta'_y(h):=\eta_{\psi(y)}({c'(h, h^{-1}y)}).\] It is straightforward to check that $L$ is the inverse of $\pi$. Indeed, this boils down to check that $c(c'(h, y), \psi(y))=h$ and $c'(c(g, x), \phi(x))=g$, which hold by Lemma \ref{lemma: xin's lemma}. Then using $L$, we know $\pi$ maps the three subspaces onto the corresponding ones.
\end{proof}

\subsection*{Observation}
In the proof of the main theorems, we will use the following observation, which shows that $N_0(G, X)$ is restriction invariant. 

Under the assumptions in the main theorems, for any given $g\in G$, let us list the elements in the finite set $c(g^{-1}, X):=\{c(g^{-1}, x)|~x\in X\}$ as $h_1^{-1},\dots, h_n^{-1}$. Then define $X_i:=\{x\in X: c(g^{-1}, x)=h_i^{-1}\}$. Clearly, $X=\sqcup_iX_i$ and each $X_i$ is clopen since $c$ is continuous. Note that $N_0(G, X)=C(X, \ell^1_0(G))=\oplus_iC(X_i, \ell^1_0(G))$. For each $\xi'\in C(X, \ell^1_0(G))$, let $\xi'|_{X_i}$ be the $i$-th component with respect to the above decomposition, i.e. $(\xi'|_{X_i})_g(x)=\xi'_g(x)$ if $x\in X_i$ and zero otherwise. Clearly, $\xi'|_{X_i}\in C(X, \ell^1_0(G))$. For each $\tau\in N_0(G, X)^*$, define $\tau_i\in  N_0(G, X)^*$ by setting $\tau_i(\xi')=\tau(\xi'|_{X_i})$. Then, we have that $\xi'=\sum_i\xi'|_{X_i}$ and $\tau=\sum_i\tau_i$. Note that the above decomposition is not $G$-equivariant, i.e. $g\tau_i\neq (g\tau)_i$.
Moreover, for any clopen subset $X_0$ of $X$, $\xi|_{X_0}\in N_0(G, X)$ if $\xi\in N_0(G, X)$, hence $\tau|_{X_0}$ is also well-defined. (By convention, $\xi'|_{\emptyset}:=0$ and $\tau|_{\emptyset}:=0$.)

Below we record two more lemmas that we use frequently later to simplify the proof. The first lemma tells us the map $\pi$ is equivariant in a local sense, and the second one tells us how the composition of $\pi$ and $L$ behaves under restrictions.

\begin{lem}\label{lem: module structure under restriction}
Let $g\in G$, $h\in H$ and $\xi\in N_0(H, Y)$. If $X_0\subseteq \{x\in X:~c(g^{-1}, x)=h^{-1}\}$ is clopen, then $(\pi(h\xi))|_{X_0}=(g\pi(\xi))|_{X_0}=g(\pi(\xi)|_{g^{-1}X_0})$, where $\pi$ is the map constructed in Lemma \ref{lemma: key lemma}.
\end{lem}
\begin{proof}
We may assume $X_0\neq \emptyset$; otherwise, the equality holds trivially as both side equals zero.
Take any $x\in X$ and $g'\in G$, we get
\begin{align*}
[(\pi(h\xi))|_{X_0}]_x(g')&=\begin{cases}
(h\xi)_{\phi(x)}(c(g', g'^{-1}x)),~\mbox{if $x\in X_0$}\\
0,~\mbox{otherwise}
\end{cases}\\
&=\begin{cases}
\xi_{h^{-1}\phi(x)}(h^{-1}c(g', g'^{-1}x)),~\mbox{if $x\in X_0$}\\
0,~\mbox{otherwise}
\end{cases}\\
&=\begin{cases}
\xi_{\phi(g^{-1}x)}(c(g^{-1}g', g'^{-1}x)),~\mbox{if $x\in X_0$}\\
0,~\mbox{otherwise}
\end{cases}\\
&=\begin{cases}
[\pi(\xi)]_{g^{-1}x}(g^{-1}g'),~\mbox{if $x\in X_0$}\\
0,~\mbox{otherwise}
\end{cases}\\
&=\begin{cases}
[g(\pi(\xi))]_x(g'),~\mbox{if $x\in X_0$}\\
0,~\mbox{otherwise}
\end{cases}\\
&=[(g\pi(\xi))|_{X_0}]_x(g').
\end{align*}
Hence, $\pi(h\xi)|_{X_0}=(g\pi(\xi))|_{X_0}$ holds. It is easy to check $(g\pi(\xi))|_{X_0}=g(\pi(\xi)|_{g^{-1}X_0})$ holds by definition.
\end{proof}

\begin{lem}\label{lem: composition of pi and L under restrictions}
Let $X_0$ and $Y_0$ be clopen subsets of $X$ and $Y$ respectively. If $\eta\in N_0(G, X)$ and $\xi\in N_0(H, Y)$, then $\pi(L(\eta)|_{Y_0})|_{X_0}=\eta|_{X_0\cap \psi(Y_0)}$ and $L(\pi(\xi)|_{X_0})|_{Y_0}=\xi|_{Y_0\cap \phi(X_0)}$.
\end{lem}
\begin{proof}
We will check the first equality below, the second one can be checked similarly.

Let $\xi=L(\eta)|_{Y_0}\in N_0(H, Y)$. Take any $g\in G$ and $x\in X$, then
\begin{align*}
&\quad (\pi(\xi)|_{X_0})_x(g)\\
&=\begin{cases}
\xi_{\phi(x)}(c(g, g^{-1}x)),   \quad \mbox{if $x\in X_0$} \\
0,\quad \mbox{otherwise}
\end{cases}\\
&\quad(\mbox{def. of restriction map and $\pi$})\\
&=\begin{cases}
\eta_x(c'(c(g, g^{-1}x),c(g, g^{-1}x)^{-1}\phi(x)))  ,\quad\mbox{if $x\in X_0\cap \psi(Y_0)$}\\
0,\quad  \mbox{otherwise}
\end{cases}\\
&\quad(\mbox{def. of $\xi$ and $L$})\\
&=\begin{cases}
\eta_x(g), \quad\mbox{if $x\in X_0\cap \psi(Y_0)$}\\
0,\quad \mbox{otherwise}
\end{cases}\\
&\quad(\mbox{cocycle identity and Lemma \ref{lemma: xin's lemma}})\\
&=(\eta|_{X_0\cap \psi(Y_0)})_x(g).
\end{align*}
Hence, $(\pi(L(\eta)|_{Y_0}))|_{X_0}=\eta|_{X_0\cap \psi(Y_0)}$.
\end{proof}

\subsection*{Notations}
Here we record some notations used in the context.
\begin{itemize}
\item $\phi: X\approx Y$ and $\psi: Y\to X$ are the homemophisms in the definition of COE.
\item $c: G\times X\to H$ and $c': H\times Y\to X$ denote the orbit cocycles.
\item $\pi$ is the isometry from $N_0(H, Y)$ to $N_0(G, X)$ as defined in Lemma \ref{lemma: key lemma}.
\item $L$ is the inverse map of $\pi$ as defined in Lemma \ref{lemma: key lemma}.
\item $c(g, X)$ denotes the finite set $\{c(g, x): x\in X\}$.
\item $c'(h, Y)$ denotes the finite set $\{c'(h, y): y\in Y\}$. 
\item $X_{g, h}$ denotes the set $\{x\in X: c(g^{-1}, x)=h^{-1}\}$. For later use, we usually simplify $X_{g_i, h}$ (resp. $X_{g, h_j}$) to $X_i$ (resp. $X_j$) if the variable $h$ (resp. $g$) is fixed or clear from the context.
\item $\eta$ denotes an element in $N_0(G, X)$ or $N_0(G, X)^{**}$ depending on the context.
\item $\xi$ denotes an element in $N_0(H, Y)$ or $N_0(H, Y)^{**}$ depending on the context. 
\item $\tau$ denotes an element in $N_0(G, X)^*$.
\item $\nu$ denotes an element in $N_0(H, Y)^*$.
\end{itemize}

\section{Proof of Theorem \ref{thm on uf homology}} \label{section: proof of uf-homology case}

By \cite[Corollary 10]{bnnw}, we know that $H_0^{uf}(G\curvearrowright X)\cong \mathbb{R}\oplus H_0(G, N_0(G, X)^*)$ or $H_0(G, N_0(G, X)^*)$ depending on whether the action is amenable or not. Therefore, to prove Theorem \ref{thm on uf homology}, it suffices to prove the second part, i.e. $H_n(G, N_0(G, X)^*)\cong H_n(H, N_0(H, Y)^*)$ for all $n\geq 0$.

We will construct a (linear) map $S_n: C_f(G^n, N_0(G, X)^*)\to C_f(H^n, N_0(H, Y)^*)$ and check $\partial_nS_n=S_{n-1}\partial_n$ for each $n$, i.e. $S_n$ is a chain map. Moreover, we will show it is an isomorphism by finding the inverse map $T_n$. Clearly, this will induce an isomorphism between homology groups and hence finish the proof.

We present the detailed proof when $n=0, 1$ to illustrate the main ideas. The proof of the general case will be given after the proof of these two cases.
\subsection{Case $n=0$}

Clearly, the map $\pi: N_0(H, Y)\to N_0(G, X)$ as used in Lemma \ref{lemma: key lemma} induces a map, denoted by $\pi^*$, on the dual spaces. 

Define $S_0: =\pi^*$. 

Let $f\in C_f(G, N_0(G, X)^*)$, we define $S_1(f)\in C_f(H, N_0(H, Y)^*)$ by setting 
$S_1(f)=f'$, where 
\[f'(h_1)(\xi):=\sum_j f(g_j)(\pi(\xi)|_{X_j}).\] 

Here $\xi\in N_0(H, Y)$, $c'(h_1, Y):=\{c'(h_1, y): y\in Y\}:=\{g_j~|~j\}, X_j=\{x\in X:~ c(g_j^{-1}, x)=h_1^{-1}\}$. Strictly speaking, it is better to use notation $X_{g_j, h_1}$ for $X_j$, but we often simplify the subscripts if no confusion arises.

Clearly, $S_0, S_1$ are well-defined. Indeed, to see $S_1$ is well-defined, observe that if $supp(f)=F$, then $supp(S_1(f))\subseteq \cup_{g\in F}c(g, X)$, which is finite once $F$ is a finite set.\\

\textbf{Step 1}: we check $\partial_1S_1=S_0\partial_1$. \\

Let $f\in C_f(G, N_0(G, X)^*)$ and $f'=S_1(f)$. Recall $\partial_1(f')=\sum_{h_1\in H}h_1^{-1}f'(h_1)-\sum_{h_1\in H}f'(h_1)$. We need to show $\partial_1S_1(f)=S_0\partial_1(f)$; equivalently, 
\[\sum_{h_1\in H}h_1^{-1}f'(h_1)-\sum_{h_1\in H}f'(h_1)=\sum_{g\in G}\pi^*(g^{-1}f(g))-\sum_{g\in G}\pi^*(f(g)).\]  

It suffices to check 
\begin{align}
\sum_{h_1\in H}h_1^{-1}f'(h_1)&=\sum_{g\in G}\pi^*(g^{-1}f(g)),\label{eq 1 for thm 1.1, n=1.}\\
\sum_{h_1\in H}f'(h_1)&=\sum_{g\in G}\pi^*(f(g))\label{eq 2 for thm 1.1, n=1}.
\end{align}
To check (\ref{eq 2 for thm 1.1, n=1}), take any $\xi\in N_0(H, Y)$, then 

\begin{align*}
\sum_{h_1}f'(h_1)(\xi)&=\sum_{h_1}\sum_jf(g_j)(\pi(\xi)|_{X_j})=\sum_g\sum_{h_1, h_1\in c(g, X)}f(g)(\pi(\xi)|_{X_{g, h_1}})\\
&=\sum_gf(g)(\pi(\xi))=\sum_{g}\pi^*(f(g))(\xi).
\end{align*}
The 2nd last equality holds since $X=\sqcup_{h_1, h_1\in c(g, X)}X_{g, h_1}$ for every $g$. 

Similarly, let us check (\ref{eq 1 for thm 1.1, n=1.}) holds below.
\begin{align*}
\sum_{h_1}(h_1^{-1}f'(h_1))(\xi)&=\sum_{h_1}f'(h_1)(h_1\xi)=\sum_{h_1}\sum_jf(g_j)(\pi(h_1\xi)|_{X_j})\\
&=\sum_g\sum_{h_1, h_1\in c(g, X)}f(g)(\pi(h_1\xi)|_{X_{g, h_1}})\\
&=\sum_g\sum_{h_1, h_1\in c(g, X)}f(g)((g\pi(\xi))|_{X_{g, h_1}})~\mbox{(by Lemma \ref{lem: module structure under restriction})}\\
&=\sum_g f(g)((g\pi(\xi))|_{\sqcup_{h_1\in c(g, X)}X_{g, h_1}})\\
&=\sum_g f(g)(g\pi(\xi))~\mbox{(as $X=\sqcup_{h_1\in c(g, X)}X_{g, h_1}$)}\\
&=\sum_g\pi^{**}(g^{-1}f(g))(\xi).
\end{align*}

\textbf{Step 2}: $S_0$ is an isomorphism.\\

Recall in the proof of Lemma \ref{lemma: key lemma}, there exists some map $L: N_0(G, X)\to N_0(H, Y)$ defined by setting $L(\eta)=\eta'$, where $\eta'_h(y):=\eta_{\psi(y)}({c'(h, h^{-1}y)})$. By symmetry, it also induces a well-defined map $T_0:=L^*: N_0(H, Y)^*\to N_0(G, X)^*$. Clearly, $\pi^*L^*=id$ and $L^*\pi^*=id$.

\subsection{Case $n=1$}

Write $c'(h_0, Y)=\{g_i|~i\}$, $c'(h_1, Y)=\{g_j|~j\}$, $c'(h_2, Y)=\{g_k|~k\}$, $c'(h, Y)=\{g_s|~s\}$ and $c'(\bar{h}, Y)=\{g_t|~t\}$. In this subsection, we will use the following notations.
\begin{align*}
X_i&:=\{x: c(g_i^{-1}, x)=h_0^{-1}\},\quad X_j:=\{x: c(g_j^{-1}, x)=h_1^{-1}\},\\
X_k&:=\{x: c(g_k^{-1}, x)=h_2^{-1}\},\quad X_s:=\{x: c(g_s^{-1}, x)=h^{-1}\}, \\
X_t&:=\{x: c(g_t^{-1}, x)=\bar{h}^{-1}\}. 
\end{align*}
Strictly speaking, it is better to use the notation $X_{g_i, h_0}, X_{g_j, h_1}, X_{g_k, h_2}, X_{g_s, h}, X_{g_t, \bar{h}}$ for $X_i, X_j, X_k, X_s, X_t$ respectively, but we often simplify the subscripts as the variables $h_0, h_1, h_2, h, \bar{h}$ are usually fixed in the context.

Let $\theta\in C_f(G^2, N_0(G, X)^*)$, we define $S_2: C_f(G^2, N_0(G, X)^*)\to C_f(H^2, N_0(H, Y)^*)$ by setting $\theta'=S_2\theta$, where for any $\xi\in N_0(H, Y)$,
\[\theta'(h_0, h_1)(\xi):=\sum_{i, j}\theta(g_i, g_j)(\pi(\xi)|_{g_iX_j\cap X_i}).\]

Note that $S_2$ is well-defined. Indeed, observe that \[supp(\theta')\subseteq \bigcup_{(g_i, g_j)\in supp(\theta)}c(g_i, X)\times c(g_j, X).\] Therefore, $supp(\theta)$ is finite implies $supp(\theta')$ is finite.

Write $f=\partial_2\theta$ and $f'=S_1f$, from the definition of $\partial_2$, we know 
\[f(g_1)=\sum_{g_0}g_0^{-1}\theta(g_0, g_1)-\sum_{g, \bar{g}, g\bar{g}=g_1}\theta(g, \bar{g})+\sum_{g_2}\theta(g_1, g_2).\]

\textbf{Step 1}: we check $S_1\partial_2=\partial_2S_2$.\\

Evaluate both sides at $\theta$ and use the above notations, we are left to show the following identity holds.
\[f'(h_1)=\sum_{h_0}h_0^{-1}\theta'(h_0, h_1)-\sum_{\substack{h, \bar{h}\\ h\bar{h}=h_1}}\theta'(h, \bar{h})+\sum_{h_2}\theta'(h_1, h_2).\]

Let $\xi\in N_0(H, Y)$, a calculation shows:
\begin{align*}
&\quad f'(h_1)(\xi)\\
&= \sum_jf(g_j)(\pi(\xi)|_{X_j})\\
&=\sum_j[\sum_{g}\theta(g, g_j)(g(\pi(\xi)|_{X_j}))-\sum_{\substack{g, \bar{g}\\ g\bar{g}=g_j}}\theta(g, \bar{g})(\pi(\xi)|_{X_j})+\sum_{g}\theta(g_j, g)(\pi(\xi)|_{X_j})].\\
\end{align*}
\begin{align*}
&\quad[\sum_{h_0}h_0^{-1}\theta'(h_0, h_1)-\sum_{\substack{h, \bar{h}\\ h\bar{h}=h_1}}\theta'(h, \bar{h})+\sum_{h_2}\theta'(h_1, h_2)](\xi)\\
  &=\sum_{h_0}\theta'(h_0, h_1)(h_0\xi)-\sum_{\substack{h, \bar{h}\\ h\bar{h}=h_1}}\theta'(h, \bar{h})(\xi)+\sum_{h_2}\theta'(h_1, h_2)(\xi) \\
  &= \sum_{h_0}\sum_{i, j}\theta(g_i, g_j)(\pi(h_0\xi)|_{g_iX_j\cap X_i})-\sum_{\substack{h, \bar{h}\\ h\bar{h}=h_1}}\sum_{s, t}\theta(g_s, g_t)(\pi(\xi)|_{g_sX_t\cap X_s})\\
  &\quad{}+\sum_{h_2}\sum_{j, k}\theta(g_j, g_k)(\pi(\xi)|_{g_jX_k\cap X_j}). \\
\end{align*}
Comparing the two expressions above, we just need to prove the corresponding terms are equal, i.e.

\begin{align}
\sum_j\sum_{g}\theta(g, g_j)(g(\pi(\xi)|_{X_j}))&=\sum_{h_0}\sum_{i, j}\theta(g_i, g_j)(\pi(h_0\xi)|_{g_iX_j\cap X_i})\label{eq6},\\
\sum_j\sum_{\substack{g, \bar{g}\\ g\bar{g}=g_j}}\theta(g, \bar{g})(\pi(\xi)|_{X_j})&=\sum_{\substack{h, \bar{h}\\ h\bar{h}=h_1}}\sum_{s, t}\theta(g_s, g_t)(\pi(\xi)|_{g_sX_t\cap X_s})\label{eq7},\\
\sum_j\sum_{g}\theta(g_j, g)(\pi(\xi)|_{X_j})&=\sum_{h_2}\sum_{j, k}\theta(g_j, g_k)(\pi(\xi)|_{g_jX_k\cap X_j})\label{eq8}.
\end{align}

To check \eqref{eq6}, apply Lemma \ref{lem: module structure under restriction} to get $\pi(h_0\xi)|_{g_iX_j\cap X_i}=g_i(\pi(\xi)|_{X_j\cap g_i^{-1}X_i})$. 

Then, RHS of \eqref{eq6} = $\sum_j\sum_{h_0, i}\theta(g_i, g_j)(g_i(\pi(\xi)|_{X_j\cap g_i^{-1}X_i}))$.

Now it suffices to prove for every $j$, 
\[\sum_{h_0, i}\theta(g_i, g_j)(g_i(\pi(\xi)|_{X_j\cap g_i^{-1}X_i}))=\sum_{g}\theta(g, g_j)(g(\pi(\xi)|_{X_j})).\]

First, recall that $g_i^{-1}X_i=\{x: c(g_i, x)=h_0\}$, which we denote by $Z_{g_i, h_0}$.

Clearly, for every $g\in G$, $X=\sqcup_{h_0\in c(g, X)}Z_{g, h_0}$, so 
\begin{align*}
\sum_{h_0, i}\theta(g_i, g_j)(g_i(\pi(\xi)|_{X_j\cap g_i^{-1}X_i}))
&=\sum_{g}\sum_{h_0\in c(g, X)}\theta(g, g_j)(g(\pi(\xi)|_{X_j\cap Z_{g, h_0}}))\\
&=\sum_{g}\theta(g, g_j)(g(\pi(\xi)|_{X_j})).
\end{align*}

To check \eqref{eq7}, first, observe that we have a bijection/reordering between the two index sets:

\begin{align*}
\left\{ (h, \bar{h}, g_s, g_t)
    \begin{tabular}{|l}
      $h\bar{h}=h_1$ \\
      $g_s\in c'(h, Y), g_t\in c'(\bar{h}, Y)$ \\
      $g_sX_{g_t, \bar{h}}\cap X_{g_s, h}\neq \emptyset$
\end{tabular}
\right\}
\end{align*}
and

\begin{align*}
\left\{ (g_j, g, \bar{g}, h, \bar{h})
\begin{tabular}{|l}
$h\bar{h}=h_1, g\bar{g}=g_j$\\
$g_j\in c'(h_1, Y), h\in c(g, X), \bar{h}\in c(\bar{g}, X)$\\
$gX_{\bar{g}, \bar{h}}\cap X_{g, h}\neq \emptyset$
\end{tabular}
\right\}.
\end{align*}

%$$\{(h, \bar{h}, g_s, g_t):~ h\bar{h}=h_1, g_s\in c'(h, Y), g_t\in c'(\bar{h}, Y), g_sX_{g_t, \bar{h}}\cap X_{g_s, h}\neq \emptyset\}$$
%and $$\{(g_j, g, \bar{g}, h, \bar{h}):~ h\bar{h}=h_1, g\bar{g}=g_j, g_j\in c'(h_1, Y), h\in c(g, X), \bar{h}\in c(\bar{g}, X), gX_{\bar{g}, \bar{h}}\cap X_{g, h}\neq \emptyset\}$$

Indeed, one can define a bijection as follows: $g_s\mapsto g, g_t\mapsto \bar{g}, h\mapsto h, \bar{h}\mapsto \bar{h}$.

Then we have the following.

\begin{align*}
\text{RHS of \eqref{eq7}}
&= \sum_{\substack{h, \bar{h}\\ h\bar{h}=h_1}}\sum_{s, t}\theta(g_s, g_t)(\pi(\xi)|_{g_sX_{g_t, \bar{h}}\cap X_{g_s, h}})\\
&= \sum_j\sum_{\substack{g, \bar{g}\\ g\bar{g}=g_j}}\sum_{\substack{h,\bar{h}, h\bar{h}=h_1\\ h\in c(g,X), \bar{h}\in c(\bar{g}, X)}}\theta(g, \bar{g})(\pi(\xi)|_{gX_{\bar{g}, \bar{h}}\cap X_{g, h}})\\
&\quad(\mbox{use the above bijection to do change of variables})\\
&= \sum_j\sum_{\substack{g, \bar{g}\\ g\bar{g}=g_j}}\theta(g, \bar{g})(\pi(\xi)|_{X_j})=\text{LHS of \eqref{eq7}}.
\end{align*}
The 2nd last equality holds since for every $g, \bar{g}$ with $g\bar{g}=g_j$, 
\[X_j=\bigsqcup_{\substack{h,\bar{h}, h\bar{h}=h_1\\ h\in c(g,X), \bar{h}\in c(\bar{g}, X)}}(gX_{\bar{g},\bar{h}}\cap X_{g, h}),\] which can be checked easily.

To check \eqref{eq8}, observe for every $j$, 
\begin{align*}
&\quad\sum_{h_2, k}\theta(g_j, g_k)(\pi(\xi)|_{g_jX_k\cap X_j})\\
&=\sum_{h_2, k}\theta(g_j, g_k)(\pi(\xi)|_{g_jX_{g_k, h_2}\cap X_j})~(\mbox{as $X_k=X_{g_k, h_2}$ by convention})\\
&=\sum_{g}\sum_{\substack{h_2\\h_2\in c(g, X)}}\theta(g_j, g)(\pi(\xi)|_{g_jX_{g, h_2}\cap X_j})\\
&=\sum_g\theta(g_j, g)(\pi(\xi)|_{X_j}).
\end{align*}
The last equality holds since $X=\sqcup_{h_2~s.t.~ h_2\in c(g, X)}g_jX_{g, h_2}$ for every $g$.\\

\textbf{Step 2}: $S_1$ is an isomorphism.\\

Recall $L: N_0(G, X)\to N_0(H, Y)$ is the inverse of $\pi$. Then we may define a map $T_1: C_f(H, N_0(H, Y)^*)\to C_f(G, N_0(G, X)^*)$ by setting $T_1(f')=f''$, where 
\[f''(g_1)(\eta):=\sum_jf'(h_j)(L(\eta)|_{Y_j}).\]
Here $\eta\in N_0(G, X)$ and  $Y_j:=\{y\in Y:~ c'(h_j^{-1}, y)=g_1^{-1}\}$.

Let $f\in C_f(G, N_0(G, X)^*)$, write $f'=S_1f$, $f''=T_1f'$. We check $T_1S_1=id$. Equivalently, $f''=f$.
A calculation shows
\begin{align*}
&\quad f''(g_1)(\eta)\\
&=\sum_jf'(h_j)(L(\eta)|_{Y_j})~(\mbox{def. of $T_1$})\\
&=\sum_j\sum_if(g_i)(\pi(L(\eta)|_{Y_j})|_{X_{g_i, h_j}})~(\mbox{def. of $S_1$})\\
&=\sum_{j, i}f(g_i)(\eta|_{X_{g_i, h_j}\cap \psi(Y_j)}).~(\mbox{Lemma \ref{lem: composition of pi and L under restrictions}})
\end{align*}
Recall by our notation, $X_{g_i, h_j}:=\{x\in X: c(g_i^{-1}, x)=h_j^{-1}\}$. 
Now, it is easy to check that $X_{g_i, h_j}\cap \psi(Y_j)=\emptyset$ unless $g_i=g_1$. And $X_{g_1, h_j}\cap\psi(Y_j)=X_{g_1, h_j}$ by Lemma \ref{lemma: xin's lemma}.

Hence, $f''(g_1)(\eta)=\sum_jf(g_1)(\eta|_{X_{g_1, h_j}})=f(g_1)(\eta)$.

So $T_1S_1=id$. By symmetry, $S_1T_1=id$.

\subsection{General case}
Based on the definition of $S_i$ for $i=0, 1$, it is natural to consider the following general formula for $S_n$.

For all $n\geq 2$, let $\theta\in C_f(G^n, N_0(G, X)^*)$, define $S_n(\theta)=\theta'\in C_f(H^n, N_0(H, Y)^*)$ by setting
for every $\xi\in N_0(H, Y)$,
\begin{align}{\label{eq for thm 1.1: def of S_n}}
\theta'(h_0,\dots,h_{n-1})(\xi):=\sum_{t_0,\dots,t_{n-1}}\theta(g_{t_0},\dots,g_{t_{n-1}})(\pi(\xi)|_{[t_0,\ldots,t_{n-1}]}).
\end{align}

Here, $c'(h_i^{-1}, Y)=\{g_{t_i}^{-1}: t_i\in T_i\}$ for some finite set $T_i$ for all $i=0,\cdots, n-1$, 
\begin{align}\label{eq for lem 1.1: def of [t0,...,tn]}
\begin{split}
X_{t_i}&:=\{x: c(g_{t_i}^{-1}, x)=h_i^{-1}\},\\
[t_0,\ldots,t_{n-1}]&:=X_{t_0}\cap g_{t_0}X_{t_1}\cap\dots\cap (g_{t_0}\cdots g_{t_{n-2}})X_{t_{n-1}}.
\end{split}
\end{align}

Strictly speaking, the index set for the above $\sum$ in  (\ref{eq for thm 1.1: def of S_n}) should be $(g_{t_0},\ldots, g_{t_{n-1}})$, but we simplify it to $(t_0,\ldots, t_{n-1})$ for convenience. Again, we will switch to the notation $X_{g_{t_i}, h_i}$ once using $X_{t_i}$ causes confusion. We also reserve the notation $g_i$ for an arbitrary element in $G$.

Clearly, $S^n$ is well-defined. Indeed, observe that
\[supp(\theta')\subseteq \bigcup_{(g_{t_0},\cdots, g_{t_{n-1}})\in supp(\theta)}c(g_{t_0}, X)\times \cdots\times c(g_{t_{n-1}}, X).\]
Hence, $supp(\theta)$ is finite implies $supp(\theta')$ is finite.

Our goal is to check that $\partial_nS_n=S_{n-1}\partial_n$ holds and $S_n$ is a bijection for each $n\geq 2$. This will finish the proof of Theorem \ref{thm on uf homology}.\\

\textbf{Step 1}: We claim $\partial_nS_n=S_{n-1}\partial_n$.\\

Take any $\theta\in C_f(G^n, N_0(G, X)^*)$, write $\theta'=S_n(\theta)$ and $f=\partial_n\theta$. 
Fix any $(h_1, \ldots, h_{n-1})\in H^{n-1}$ and any $\xi\in N_0(H, Y)$.

Now, we calculate both $[\partial_nS_n(\theta)](h_1, \ldots, h_{n-1})(\xi)
$ and $[S_{n-1}\partial_n(\theta)](h_1, \ldots, h_{n-1})(\xi)$ and show they are equal.

First, we have the following calculation.
\begin{align*}
&\quad [\partial_nS_n(\theta)](h_1, \ldots, h_{n-1})\\
&=[\partial_n\theta'](h_1, \ldots, h_{n-1})\\
&=\sum_{h_0\in H}h_0^{-1}\theta'(h_0,\ldots, h_{n-1})+\sum_{i=1}^{n-1}(-1)^i\sum_{\substack{h, \bar{h}\in H\\ h\bar{h}=h_i}}\theta'(h_1,\ldots, h_{i-1},h,\bar{h}, h_{i+1},\ldots, h_{n-1})\\
&\quad{} +(-1)^n\sum_{h_n\in H}\theta'(h_1,\ldots, h_n).
\end{align*}
Evaluate both sides of the above at the element $\xi$, then plug in the definition of $\theta'$, i.e. (\ref{eq for thm 1.1: def of S_n}) into the RHS of the above, we deduce that
\begin{align}\label{eq for thm 1.1: LHS of chain equality}
\begin{split}
&\quad\quad [\partial_nS_n(\theta)](h_1, \ldots, h_{n-1})(\xi)\\
&=\sum_{h_0\in H} \sum_{t_0,\ldots,t_{n-1}}\big \langle\theta(g_{t_0},\ldots, g_{t_{n-1}}), \pi(h_0\xi)|_{[t_0, \ldots, t_{n-1}]} \big\rangle\\
&\quad+\sum_{i=1}^{n-1}(-1)^i\sum_{\substack{h, \bar{h}\in H\\ h\bar{h}=h_i}} \sum_{\substack{t_1,\ldots, t_{i-1}, s\\ k, t_{i+1},\ldots, t_{n-1}}}\big\langle\theta(g_{t_1},\ldots, g_{t_{i-1}}, g_s, g_k, g_{t_{i+1}},\ldots, g_{t_{n-1}}),\\
&\qquad\qquad\qquad\qquad \pi(\xi)|_{[t_1,\ldots, t_{i-1}, s, k, t_{i+1}, \ldots, t_{n-1}]}\big \rangle\\
&\quad+(-1)^n\sum_{h_n\in H} \sum_{t_1,\ldots, t_n}\big\langle\theta(g_{t_1},\ldots, g_{t_n}), \pi(\xi)|_{[t_1,\ldots, t_n]} \big\rangle.
\end{split}
\end{align}
Here, we use $\langle -, - \rangle$ to denote the evaluation, i.e. evaluating the first entry (a function) at the second entry (a variable).
Let us recall some notations used here:
\begin{align*}
X_s:=\{x: c(g_s^{-1}, x)=h^{-1}\},\quad X_k:=\{x: c(g_k^{-1}, x)=\bar{h}^{-1}\}.
\end{align*}
\begin{multline}\label{eq for thm 1.1: def of [t1,...s,k,...]}
[t_1,\ldots, t_{i-1}, s, k, t_{i+1}, \ldots, t_{n-1}]:=\\
 X_{t_1}\cap g_{t_1}X_{t_2}\cap\cdots\cap (g_{t_1}\cdots g_{t_{i-2}})X_{t_{i-1}}\cap (g_{t_1}\cdots g_{t_{i-1}})X_s\\
\cap (g_{t_1}\cdots g_{t_{i-1}}g_s)X_k\cap\cdots\cap(g_{t_{1}}\cdots g_{t_{i-1}}g_sg_kg_{t_{i+1}}\cdots g_{t_{n-2}})X_{t_{n-1}}.
\end{multline}
Second, another calculation tells us the following.
\begin{align}\label{eq for thm 1.1: prestage of RHS of chain equality}
\begin{split}
&\quad [S_{n-1}\partial_n\theta](h_1, \ldots, h_{n-1})(\xi)\\
&=[S_{n-1}f](h_1, \ldots, h_{n-1})(\xi)\\
&=\sum_{t_1,\ldots, t_{n-1}}f(g_{t_1},\ldots, g_{t_{n-1}})(\pi(\xi)|_{[t_1,\ldots, t_{n-1}]}).\quad(\mbox{def. of $S_{n-1}$})
\end{split}
\end{align}
Now, recall that $f=\partial_n\theta$ and the definition of $\partial_n$ from Section \ref{section: preliminaries}, we get
\begin{align*}
\quad f(g_{t_1},\ldots, g_{t_{n-1}})
&=\sum_{g_0\in G}g_0^{-1}\theta(g_0, g_{t_1},\ldots, g_{t_{n-1}})\\
&\quad+\sum_{i=1}^{n-1}(-1)^i\sum_{\substack{g, \bar{g}\in G\\g\bar{g}=g_{t_i}}}\theta(g_{t_1},\ldots, g_{t_{i-1}}, g, \bar{g}, g_{t_{i+1}},\ldots, g_{t_{n-1}})\\
&\quad+(-1)^n\sum_{g_n\in G}\theta(g_{t_1},\ldots, g_{t_{n-1}}, g_n).
\end{align*}

Evaluate both sides of the above at $\pi(\xi)|_{[t_1,\ldots, t_{n-1}]}$ and combine with (\ref{eq for thm 1.1: prestage of RHS of chain equality}), we get 

\begin{align}\label{eq for thm 1.1: RHS of chain equality}
\begin{split}
& \quad [S_{n-1}\partial_n\theta](h_1, \ldots, h_{n-1})(\xi)\\
&=\sum_{t_1, \ldots, t_{n-1}}\sum_{g_0\in G}\langle g_0^{-1}\theta(g_0, g_{t_1},\ldots, g_{t_{n-1}}), \pi(\xi)|_{[t_1,\ldots, t_{n-1}]}\rangle\\
&\quad+\sum_{t_1,\ldots, t_{n-1}}\sum_{i=1}^{n-1}(-1)^i\sum_{\substack{g, \bar{g}\in G\\g\bar{g}=g_{t_i}}}\big\langle \theta(g_{t_1},\ldots, g_{t_{i-1}}, g, \bar{g}, g_{t_{i+1}},\ldots, g_{t_{n-1}}), \pi(\xi)|_{[t_1,\ldots, t_{n-1}]}\big\rangle\\
&\quad+\sum_{t_1,\ldots, t_{n-1}}(-1)^n\sum_{g_n\in G}\big\langle \theta(g_{t_1},\ldots, g_{t_{n-1}}, g_n), \pi(\xi)|_{[t_1,\ldots, t_{n-1}]}\big \rangle.
\end{split}
\end{align}
We aim to show (\ref{eq for thm 1.1: LHS of chain equality})=(\ref{eq for thm 1.1: RHS of chain equality}). To do this, we just need to compare all the corresponding summands and prove they are equal, i.e. it suffices to check the following hold. 

(a) \begin{multline*}
 \sum_{h_0\in H} \sum_{t_0,\ldots,t_{n-1}}\big \langle\theta(g_{t_0},\ldots, g_{t_{n-1}}), \pi(h_0\xi)|_{[t_0, \ldots, t_{n-1}]} \big\rangle\\
=\sum_{t_1, \ldots, t_{n-1}}\sum_{g_0\in G}\langle g_0^{-1}\theta(g_0, g_{t_1},\ldots, g_{t_{n-1}}), \pi(\xi)|_{[t_1,\ldots, t_{n-1}]}\rangle.
\end{multline*}

%(b)\begin{align*}
%&\sum_{h_0\in H} \sum_{t_0,\ldots,t_{n-1}}\big\langle\theta(g_{t_0},\ldots, g_{t_{n-1}}), \pi(h_0\xi)|_{[t_0, \ldots, t_{n-1}]}\big \rangle\\
%&=\sum_{t_1, \ldots, t_{n-1}}\sum_{g_0\in G}\big\langle (g_0^{-1}\theta(g_0, g_{t_1},\ldots, g_{t_{n-1}}), \pi(\xi)|_{[t_1,\ldots, t_{n-1}]}\big\rangle.
%\end{align*}

(b) For any $1\leq i\leq n-1$, 
\begin{multline*}
\sum_{\substack{h, \bar{h}\in H\\ h\bar{h}=h_i}} \sum_{\substack{t_1,\ldots, t_{i-1}, s\\ k, t_{i+1},\ldots, t_{n-1}}}\big\langle\theta(g_{t_1},\ldots, g_{t_{i-1}}, g_s, g_k, g_{t_{i+1}},\ldots, g_{t_{n-1}}), \pi(\xi)|_{[t_1,\ldots, t_{i-1}, s, k, t_{i+1}, \ldots, t_{n-1}]} \big\rangle\\
=\sum_{t_1,\ldots, t_{n-1}}\sum_{\substack{g, \bar{g}\in G\\g\bar{g}=g_{t_i}}}\big\langle \theta(g_{t_1},\ldots, g_{t_{i-1}}, g, \bar{g}, g_{t_{i+1}},\ldots, g_{t_{n-1}}), \pi(\xi)|_{[t_1,\ldots, t_{n-1}]}\big\rangle.
\end{multline*}

(c)\begin{multline*}
\sum_{h_n\in H} \sum_{t_1,\ldots, t_n}\big\langle\theta(g_{t_1},\ldots, g_{t_n}), \pi(\xi)|_{[t_1,\ldots, t_n]} \big\rangle\\
=\sum_{t_1,\ldots, t_{n-1}}\sum_{g_n\in G}\big\langle \theta(g_{t_1},\ldots, g_{t_{n-1}}, g_n), \pi(\xi)|_{[t_1,\ldots, t_{n-1}]}\big \rangle.
\end{multline*}

We prove (a), (b), (c) below respectively.

\begin{proof}[Proof of (a)]
By Lemma \ref{lem: module structure under restriction}, $\pi(h_0\xi)|_{[t_0,\ldots, t_{n-1}]}=g_{t_0}(\pi(\xi)|_{g_{t_0}^{-1}[t_0,\ldots, t_{n-1}]}).$
Then,
\begin{align}\label{eq for thm 1.1: check (a) in general case}
\begin{split}
& \quad\quad \mbox{Top expression in (a)}\\
&=\quad \sum_{h_0\in H}\sum_{t_0,\ldots,t_{n-1}}\big \langle \theta(g_{t_0},\ldots, g_{t_{n-1}}), \pi(h_0\xi)|_{[t_0, \ldots, t_{n-1}]} \big\rangle\\
&=\quad \sum_{h_0\in H} \sum_{t_0,\ldots,t_{n-1}}\big \langle\theta(g_{t_0},\ldots, g_{t_{n-1}}), g_{t_0}(\pi(\xi)|_{g_{t_0}^{-1}[t_0,\ldots, t_{n-1}]}) \big\rangle\\
&=\quad \sum_{h_0\in H}\sum_{t_0,\ldots,t_{n-1}}\big \langle g_{t_0}^{-1}\theta(g_{t_0},\ldots, g_{t_{n-1}}), \pi(\xi)|_{g_{t_0}^{-1}X_{t_0}\cap [t_1,\ldots, t_{n-1}]} \big\rangle\\
&\quad(\mbox{def. of dual module and by (\ref{eq for lem 1.1: def of [t0,...,tn]}), $g_{t_0}^{-1}[t_0,\ldots, t_{n-1}]=g_{t_0}^{-1}X_{t_0}\cap [t_1,\ldots, t_{n-1}]$})\\
&=\quad \sum_{t_1,\ldots,t_{n-1}}\sum_{h_0\in H}\sum_{t_0}\big \langle g_{t_0}^{-1}\theta(g_{t_0},\ldots, g_{t_{n-1}}), \pi(\xi)|_{g_{t_0}^{-1}X_{t_0}\cap [t_1,\ldots, t_{n-1}]} \big\rangle\\
&\quad(\mbox{switching sums is possible as $t_1,\ldots, t_{n-1}$ and $h_0$ are independent})\\
&= \quad \sum_{t_1,\ldots,t_{n-1}}\sum_{h_0\in H}\sum_{\substack{g_{t_0}\\ g_{t_0}\in c'(h_0, Y)}}\big \langle g_{t_0}^{-1}\theta(g_{t_0},\ldots, g_{t_{n-1}}), \pi(\xi)|_{g_{t_0}^{-1}X_{g_{t_0}, h_0}\cap [t_1,\ldots, t_{n-1}]} \big\rangle\\
& \quad(\mbox{notation convention, check the explanation given when defing $S_n$, i.e. (\ref{eq for thm 1.1: def of S_n})})\\
&=\quad \sum_{t_1,\ldots,t_{n-1}}\sum_{g_0\in G}\sum_{h_0\in c(g_0, X)}\big \langle g_0^{-1}\theta(g_0,\ldots, g_{t_{n-1}}), \pi(\xi)|_{g_0^{-1}X_{g_0, h_0}\cap [t_1,\ldots, t_{n-1}]} \big\rangle.
\end{split}
\end{align}
The last equality holds for the following reason:

The map $h_0\mapsto h_0, g_{t_0}\mapsto g_0$ induces a bijection between the two index sets: $\{(h_0, g_{t_0}): h_0\in H, g_{t_0}\in c'(h_0, Y)\}$ and $\{(g_0, h_0): g_0\in G, h_0\in c(g_0, X)\}$. This explains the change of the index sets under the two rightmost sums. So under the above bijection, $g_{t_0}^{-1}X_{g_{t_0}, h_0}$ is replaced by $g_0^{-1}X_{g_0, h_0}$. 

Now, recall $x\in g_0^{-1}X_{g_0, h_0}$ iff $g_0x\in X_{g_0, h_0}$ iff $c(g_0, x)=h_0$. 
Therefore, for each $g_0\in G$, $X=\sqcup_{h_0\in c(g_0, X)}g_0^{-1}X_{g_0, h_0}$.

Using this fact, we can continue to simplify (\ref{eq for thm 1.1: check (a) in general case}) to the following expressions.
\begin{align*}
&\qquad\mbox{The top expression in (a)}\\
&= \sum_{t_1,\ldots,t_{n-1}}\sum_{g_0\in G}\big \langle g_0^{-1}\theta(g_0,\ldots, g_{t_{n-1}}), \pi(\xi)|_{(\sqcup_{h_0\in c(g_0, X)}g_0^{-1}X_{g_0, h_0})\cap [t_1,\ldots, t_{n-1}]} \big\rangle\\
&=\sum_{t_1,\ldots,t_{n-1}}\sum_{g_0\in G}\big \langle g_0^{-1}\theta(g_0,\ldots, g_{t_{n-1}}), \pi(\xi)|_{X\cap [t_1,\ldots, t_{n-1}]} \big\rangle\\
&= \sum_{t_1,\ldots,t_{n-1}}\sum_{g_0\in G}\big \langle g_0^{-1}\theta(g_0,\ldots, g_{t_{n-1}}), \pi(\xi)|_{ [t_1,\ldots, t_{n-1}]} \big\rangle\\
&=\quad\mbox{The bottom expression in (a)}.\qedhere
\end{align*}
\end{proof}

\begin{proof}[Proof of (b)]

The proof relies on the following two facts:\\

(Fact 1) For every fixed $t_1,\ldots, t_{i-1}, t_{i+1},\ldots, t_{n-1}$, the map $t_j\mapsto t_j$ ($\forall j\neq i$), $s\mapsto g$, $k\mapsto \bar{g}$ induces a bijection between the index sets:

%$\{(t_1,\ldots, t_{i-1}, t_{i+1},\ldots, t_{n-1}, s, k): h\bar{h}=h_i, [t_1,\ldots, t_{i-1}, s, k, t_{i+1}, \ldots, t_{n-1}]\neq \emptyset, g_s\in c'(h, Y), g_k\in c'(\bar{h}, Y) \}$ and $\{(t_1,\ldots, t_{n-1}, g, \bar{g}): h\in c(g, X), \bar{h}\in c(\bar{g}, X), g\bar{g}=g_{t_i}, h\bar{h}=h_i, g_{t_i}\in c'(h_i, Y), [t_1,\ldots, t_{n-1}]\neq \emptyset\}.$

\begin{align*}
\left\{  (t_1,\ldots, t_{i-1}, t_{i+1},\ldots, t_{n-1}, s, k) ~   
\begin{tabular}{|l}
$h\bar{h}=h_i$\\            
$g_s\in c'(h, Y), g_k\in c'(\bar{h}, Y)$\\
$[t_1,\ldots, t_{i-1}, s, k, t_{i+1}, \ldots, t_{n-1}]\neq \emptyset$
\end{tabular}
\right\}
\end{align*}
and 
\begin{align*}
\left\{  
(t_1,\ldots, t_{n-1}, g, \bar{g})~
\begin{tabular}{|l}
$g\bar{g}=g_{t_i}, h\bar{h}=h_i$\\
$ h\in c(g, X), \bar{h}\in c(\bar{g}, X), g_{t_i}\in c'(h_i, Y)$\\
$[t_1,\ldots, t_{i-1}, g,\bar{g}, t_{i+1},\ldots, t_{n-1}]\neq \emptyset$
\end{tabular}
\right\}.
\end{align*}

Indeed, just remember that the index $s$ stands for $g_s$, $k$ stands for $g_k$, $t_j$ stands for $g_{t_j}$ and $g_{t_i}$ is determined by $g$ and $\bar{g}$.  \\

Here, $[t_1,\ldots, t_{i-1}, g,\bar{g},t_{i+1},\ldots, t_{n-1}]$ is the image of $[t_1,\ldots, t_{i-1}, s, k, t_{i+1}, \ldots, t_{n-1}]$ under the bijection as in Fact 1. More precisely, from the def. (\ref{eq for thm 1.1: def of [t1,...s,k,...]}), we know

\begin{multline}\label{eq for thm 1.1: def of [t1,...g, g bar, ...]}
 [t_1,\ldots, t_{i-1}, g,\bar{g},t_{i+1},\ldots, t_{n-1}]=
 X_{t_1}\cap g_{t_1}X_{t_2}\cap\cdots\cap (g_{t_1}\cdots g_{t_{i-2}})X_{t_{i-1}}\\
\cap (g_{t_1}\cdots g_{t_{i-1}})X_g\cap (g_{t_1}\cdots g_{t_{i-1}}g)X_{\bar{g}}\cap(g_{t_{1}}\cdots g_{t_{i-1}}g\bar{g})X_{t_{i+1}}\\
\cap(g_{t_{1}}\cdots g_{t_{i-1}}g\bar{g}g_{t_{i+1}})X_{t_{i+2}} \cap \cdots\cap(g_{t_{1}}\cdots g_{t_{i-1}}g\bar{g}g_{t_{i+1}}\cdots g_{t_{n-2}})X_{t_{n-1}}.
\end{multline}

We remind the reader in the above expression, $X_g=X_{g, h}=\{x: c(g^{-1}, x)=h^{-1}\}$ and $X_{\bar{g}}=\{x: c(\bar{g}^{-1}, x)=\bar{h}^{-1}\}$.\\

(Fact 2) For all $g, \bar{g}\in G$ satisfying $g\bar{g}=g_{t_i}$, we have \begin{align*}
&\bigsqcup_{\substack{h,\bar{h}, h\bar{h}=h_i\\h\in c(g, X), \bar{h}\in c(\bar{g}, X)}}(g_{t_1}\cdots g_{t_{i-1}})X_{g, h}\cap (g_{t_1}\cdots g_{t_{i-1}}g)X_{\bar{g},\bar{h}}\cap (g_{t_1}\cdots g_{t_{i-1}}g\bar{g})X_{t_{i+1}}\\
&\quad{}=(g_{t_1}\cdots g_{t_{i-1}})X_{t_i}\cap (g_{t_1}\cdots g_{t_i})X_{t_{i+1}}. 
\end{align*}

Recall $X_{t_{i+1}}=X_{g_{t_{i+1}}, h_{i+1}}=\{x: c(g_{t_{i+1}}^{-1}, x)=h_{i+1}^{-1}\}$, $X_{g, h}=\{x: c(g^{-1}, x)=h^{-1}\}$ and $X_{\bar{g}, \bar{h}}=\{x: c(\bar{g}^{-1}, x)=\bar{h}^{-1}\}$.
Then Fact 2 is easy to verify using cocycle identities.

Now, we can prove (b) as follows.

By Fact 1,
\begin{align}\label{eq for thm 1.1: top line in (b)}
\begin{split}
&\quad \mbox{The top expression in (b)}\\
&=\sum_{t_1,\ldots, t_{n-1}}\sum_{\substack{g, \bar{g}\in G\\g\bar{g}=g_{t_i}}}\sum_{\substack{h,\bar{h}, h\bar{h}=h_i\\h\in c(g, X), \bar{h}\in c(\bar{g}, X)}}\big\langle \theta(g_{t_1},\ldots, g_{t_{i-1}}, g, \bar{g}, g_{t_{i+1}},\ldots, g_{t_{n-1}}), \\
& \qquad\qquad\qquad\qquad\qquad\qquad\qquad\pi(\xi)|_{[t_1,\ldots, t_{i-1}, g,\bar{g},t_{i+1},\ldots, t_{n-1}]}\big\rangle.
\end{split}
\end{align}
 Since $g\bar{g}=g_{t_i}$, we can rewrite
(\ref{eq for thm 1.1: def of [t1,...g, g bar, ...]}) as 
\begin{multline}\label{eq for thm 1.1: final form for [t1,...,g, g bar,...]}
 [t_1,\ldots, t_{i-1}, g,\bar{g},t_{i+1},\ldots, t_{n-1}]=
 X_{t_1}\cap g_{t_1}X_{t_2}\cap\cdots\cap (g_{t_1}\cdots g_{t_{i-2}})X_{t_{i-1}}\\
\cap \underline{(g_{t_1}\cdots g_{t_{i-1}})X_{g, h}\cap (g_{t_1}\cdots g_{t_{i-1}}g)X_{\bar{g}, \bar{h}}\cap(g_{t_{1}}\cdots g_{t_{i-1}}g\bar{g})X_{t_{i+1}}}\\
\cap(g_{t_{1}}\cdots g_{t_{i-1}}g_{t_i}g_{t_{i+1}})X_{t_{i+2}} \cap \cdots\cap(g_{t_{1}}\cdots g_{t_{i-1}}g_{t_i}g_{t_{i+1}}\cdots g_{t_{n-2}})X_{t_{n-1}}.
\end{multline}
Now, observe that the underlined part above is exactly the one appeared in Fact 2, and the rest pieces in $(\ref{eq for thm 1.1: final form for [t1,...,g, g bar,...]})$ are exactly the ones appeared in the definition of $[t_1,\ldots,t_{n-1}]$, which is recalled below
\begin{multline*}
[t_1,\ldots, t_{n-1}]=
 X_{t_1}\cap g_{t_1}X_{t_2}\cap\cdots\cap (g_{t_1}\cdots g_{t_{i-2}})X_{t_{i-1}}\\
\cap \underline{(g_{t_1}\cdots g_{t_{i-1}})X_{t_i}\cap (g_{t_1}\cdots g_{t_{i}})X_{t_{i+1}}}\\
\cap(g_{t_{1}}\cdots g_{t_{i-1}}g_{t_i}g_{t_{i+1}})X_{t_{i+2}} \cap \cdots\cap(g_{t_{1}}\cdots g_{t_{i-1}}g_{t_i}g_{t_{i+1}}\cdots g_{t_{n-2}})X_{t_{n-1}}.
\end{multline*}

Apply the above observation, we can continue the computation in (\ref{eq for thm 1.1: top line in (b)}) as follows.

First, we can move \[\sum_{\substack{h,\bar{h}, h\bar{h}=h_i\\h\in c(g, X), \bar{h}\in c(\bar{g}, X)}}\] inside $\langle -,- \rangle$ and put it before the 2nd entry, i.e. $\pi(\xi)|_{[t_1,\ldots, t_{i-1}, g,\bar{g},t_{i+1},\ldots, t_{n-1}]}$, to get

\begin{align*}
\sum_{\substack{h,\bar{h}, h\bar{h}=h_i\\h\in c(g, X), \bar{h}\in c(\bar{g}, X)}}\pi(\xi)|_{[t_1,\ldots, t_{i-1}, g,\bar{g},t_{i+1},\ldots, t_{n-1}]}=\pi(\xi)|_{\mathcal{J}}.
\end{align*}
Here, 
\begin{align*}
\mathcal{J}&:=\bigsqcup_{\substack{h,\bar{h}, h\bar{h}=h_i\\h\in c(g, X), \bar{h}\in c(\bar{g}, X)}}[t_1,\ldots, t_{i-1}, g,\bar{g},t_{i+1},\ldots, t_{n-1}]\\
&=[t_1,\ldots,t_{n-1}].\qquad(\mbox{by Fact ~2~+~(\ref{eq for thm 1.1: final form for [t1,...,g, g bar,...]})})
\end{align*}
Put all the above information together, we finally arrive at the following.
\begin{align*}
&\quad\mbox{The top expression in (b)}\\
&=\sum_{t_1,\ldots, t_{n-1}}\sum_{\substack{g, \bar{g}\in G\\g\bar{g}=g_{t_i}}}\big\langle \theta(g_{t_1},\ldots, g_{t_{i-1}}, g, \bar{g}, g_{t_{i+1}},\ldots, g_{t_{n-1}}), \pi(\xi)|_{[t_1,\ldots, t_{n-1}]}\big\rangle\\
&=\mbox{The bottom expression in (b)}.
\end{align*}
This finishes the proof of (b).
\end{proof}

\begin{proof}[Proof of (c)]

\begin{align*}
&\quad\mbox{The top expression in (c)}\\
&= \sum_{t_1,\ldots, t_{n-1}}\sum_{h_n\in H}\sum_{t_n}\big\langle\theta(g_{t_1},\ldots, g_{t_n}), \pi(\xi)|_{[t_1,\ldots, t_n]} \big\rangle\\
&=\sum_{t_1,\ldots, t_{n-1}}\sum_{h_n\in H}\sum_{t_n}\big\langle\theta(g_{t_1},\ldots, g_{t_n}), \pi(\xi)|_{[t_1,\ldots, t_{n-1}]\cap (g_{t_1}\cdots g_{t_{n-1}})X_{t_n}} \big\rangle\\
&\qquad\mbox{(def. of $[t_1,\ldots, t_n]$)}\\
&=\sum_{t_1,\ldots, t_{n-1}}\sum_{g_n\in G}\sum_{h_n\in c(g_n, X)}\big\langle\theta(g_{t_1},\ldots ,g_{t_{n-1}}, g_n), \pi(\xi)|_{[t_1,\ldots, t_{n-1}]\cap (g_{t_1}\cdots g_{t_{n-1}})X_{g_n, h_n}} \big\rangle.
\end{align*}
The last equality holds for the following reason.

Recall here $X_{t_n}=X_{g_{t_n}, h_n}=\{x: c(g_{t_n}^{-1}, x)=h_n^{-1}\}$, so $X_{t_n}$ is replaced by $X_{g_n, h_n}$ under the bijective map $t_n\mapsto g_n$, $h_n\mapsto h_n$ between the index sets \[\{(h_n, t_n): t_n\in c'(h_n, Y), h_n\in H\}\] and \[\{(g_n, h_n): h_n\in c(g_n, X), g_n\in G\}.\]

Then, since \[X=\bigsqcup_{h_n\in c(g_n, X)}(g_{t_1}\cdots g_{t_{n-1}})X_{g_n, h_n},\]
we can continue the computation of the above expression as follows.
\begin{align*}
&\quad\mbox{The top expression in (c)}\\
&=\sum_{t_1,\ldots, t_{n-1}}\sum_{g_n\in G}\big\langle\theta(g_{t_1},\ldots, g_{t_{n-1}}, g_n), \pi(\xi)|_{[t_1,\ldots, t_{n-1}]\cap \sqcup_{h_n\in c(g_n, X)}(g_{t_1}\cdots g_{t_{n-1}})X_{g_n, h_n}} \big\rangle\\
&=\sum_{t_1,\ldots, t_{n-1}}\sum_{g_n\in G}\big\langle\theta(g_{t_1},\ldots, g_{t_{n-1}}, g_n), \pi(\xi)|_{[t_1,\ldots, t_{n-1}]}\big\rangle\\
&=\mbox{The bottom expression in (c).}\qedhere
\end{align*} 
\end{proof}

\textbf{Step 2}: $S_n$ is a bijection.\\

By symmetry, we define a (linear) map $T_n: C_f(H^n, N_0(H, Y)^*)\to C_f(G^n, N_0(G, X)^*)$ by setting $T_n(\theta')=\theta''$, where 

\begin{align}\label{eq for thm 1.1: def of T_n}
\theta''(g_0,\ldots, g_{n-1})(\eta):=\sum_{s_0,\ldots, s_{n-1}}\theta'(h_{s_0},\ldots, h_{s_{n-1}})(L(\eta)|_{[s_0,\ldots, s_{n-1}]}).
\end{align}
Here, $\eta\in N_0(G, X)$, $Y_{s_i}:=\{y\in Y: ~ c'(h_{s_i}^{-1}, y)=g_i^{-1}\}$ and 
\begin{align}\label{def for thm 1.1: def of [s_0,..., s_n-1]}
[s_0,\ldots, s_{n-1}]:=Y_{s_0}\cap h_{s_0}Y_{s_1}\cap\cdots \cap (h_{s_0}\cdots h_{s_{n-2}})Y_{s_{n-1}}.
\end{align}

If we can prove the following claim, then by symmetry, we also have $S_nT_n=id$, and hence $S_n$ is a bijection.\\

\textbf{Claim: $T_nS_n=id$.}
\begin{proof}[Proof of the claim]

Take any $\theta\in C_f(G^n, N_0(G, X)^*)$, $(g_0,\ldots, g_{n-1})\in G^n$ and any $\eta\in N_0(G, X)$, let $\theta'=S_n(\theta)$ and $\theta''=T_n(\theta')$. 

We aim to show $\theta=\theta''$, i.e. $\theta''(g_0,\ldots, g_{n-1})(\eta)=\theta(g_0,\ldots, g_{n-1})(\eta)$.

From the definitions of $S_n$ and $T_n$, i.e. (\ref{eq for thm 1.1: def of S_n}) and (\ref{eq for thm 1.1: def of T_n}), we deduce that
\begin{align}\label{eq for thm 1.1: expression for theta''}
\begin{split}
&\quad\theta''(g_0,\ldots, g_{n-1})(\eta)\\
&=\sum_{s_0,\ldots, s_{n-1}}\theta'(h_{s_0},\ldots, h_{s_{n-1}})(L(\eta)|_{[s_0,\ldots, s_{n-1}]})\\
&=\sum_{s_0,\ldots, s_{n-1}}\sum_{t_0,\ldots, t_{n-1}}\theta(g_{t_0},\ldots, g_{t_{n-1}})(\pi(L(\eta)|_{[s_0,\ldots, s_{n-1}]})|_{[t_0,\ldots, t_{n-1}]}).
\end{split}
\end{align}
Since there are many variables here, let us recall some notations used above.

$[s_0,\ldots, s_{n-1}]$ is defined as in (\ref{def for thm 1.1: def of [s_0,..., s_n-1]}), and  def. (\ref{eq for lem 1.1: def of [t0,...,tn]}) in our context is the following.
\begin{align}\label{eq for thm 1.1: def of [t_0,..,t_n-1] when checking 1-1}
[t_0, \ldots, t_{n-1}]=X_{g_{t_0}, h_{s_0}}\cap g_{t_0}X_{g_{t_1}, h_{s_1}}\cap\cdots\cap (g_{t_0}\ldots g_{t_{n-2}})X_{g_{t_{n-1}}h_{s_{n-1}}},
\end{align}
where $X_{g_{t_i}, h_{s_i}}=\{x: c(g_{t_i}^{-1}, x)=h_{s_i}^{-1}\}.$\\

To continue the computation, we need the following facts.\\

(Fact 3) $\pi(L(\eta)|_{[s_0,\ldots, s_{n-1}]})|_{[t_0,\ldots, t_{n-1}]}=\eta|_{[t_0,\ldots, t_{n-1}]\cap \psi([s_0,\ldots, s_{n-1}])}.$

This is clear by Lemma \ref{lem: composition of pi and L under restrictions}.

(Fact 4) $[t_0,\ldots, t_{n-1}]\cap \psi([s_0,\ldots, s_{n-1}])=\emptyset$ unless $g_{t_i}=g_i$ for all $0\leq i\leq n-1$. When these conditions hold, the intersection equals 
\[X_{g_0, s_0}\cap g_0X_{g_1, s_1}\cap\cdots\cap (g_0\cdots g_{n-2})X_{g_{n-1}, s_{n-1}},\] where $X_{g_i, s_i}=\{x: c(g_i^{-1}, x)=h_{s_i}^{-1}\}$.

This fact can be checked using (\ref{def for thm 1.1: def of [s_0,..., s_n-1]}), (\ref{eq for thm 1.1: def of [t_0,..,t_n-1] when checking 1-1}) and Lemma \ref{lemma: xin's lemma}.\\

Now, we can continue the computation of (\ref{eq for thm 1.1: expression for theta''}) as follows.
\begin{align*}
&\quad\theta''(g_0,\ldots, g_{n-1})(\eta)\\
&\overset{Fact~3}{=}\sum_{s_0,\ldots, s_{n-1}}\sum_{t_0,\ldots, t_{n-1}}\theta(g_{t_0},\ldots, g_{t_{n-1}})(\eta|_{[t_0,\ldots, t_{n-1}]\cap \psi([s_0,\ldots, s_{n-1}])})\\
&\overset{Fact~4}{=}\sum_{s_0,\ldots, s_{n-1}}\theta(g_0,\ldots, g_{n-1})(\eta|_{X_{g_0, s_0}\cap g_0X_{g_1, s_1}\cap\cdots\cap (g_0\cdots g_{n-2})X_{g_{n-1}, s_{n-1}}})\\
&=\theta(g_0,\ldots, g_{n-1})(\eta|_{\bigsqcup_{s_0,\ldots, s_{n-1}}(X_{g_0, s_0}\cap g_0X_{g_1, s_1}\cap\cdots\cap (g_0\cdots g_{n-2})X_{g_{n-1}, s_{n-1}})})\\
&=\theta(g_0,\ldots, g_{n-1})(\eta).
\end{align*}
The last equality holds as 
\[X=\bigsqcup_{s_0,\ldots, s_{n-1}}(X_{g_0, s_0}\cap g_0X_{g_1, s_1}\cap\cdots\cap (g_0\cdots g_{n-2})X_{g_{n-1}, s_{n-1}}),\]
which can be checked directly.
\end{proof}

\section{Proof of Theorem \ref{thm on b cohomology}}\label{section: proof of b-cohomology case}

The idea for the proof is similar to the previous one: For each $n$, we will construct a (linear) map $S^n$ from $C_b(H^n, N_0(H, Y)^{**})$ to $C_b(G^n, N_0(G, X)^{**})$ and check directly $S^n$ is a cochain map, i.e. $\partial^{n-1}S^{n-1}=S^n\partial^{n-1}$ holds. Then, we use symmetry to find the inverse of $S^n$. Since the proof of general case uses cumbersome notations, we include the proof of the initial cases ($n=1, 2, 3$) to illustrate the main ideas.

\subsection{Case $n=1$}
Recall that $C_b(H^0, N_0(H, Y)^{**})=N_0(H, Y)^{**}$. We define $S^0=\pi^{**}$ and $S^1(f)=f'$, where 
\[f'(g)(\tau):=\sum_i\langle \pi^{**}(f(h_i)),\tau_i \rangle.\]
Here, $f\in  C_b(H, N_0(H, Y)^{**})$, $\tau\in N_0(G, X)^*$ and $\tau_i=\tau|_{X_i}$. Moreover, we write $X_i=\{x\in X:~ c(g^{-1}, x)=h_i^{-1}\}$, where $c(g, X)=\{h_i|~i\}$. \\

\textbf{Step 1: check $S^1$ is well-defined, i.e. $\sup_{g\in G}||f'(g)||<\infty$.}\\

Since $\langle \pi^{**}(f(h_i)), \tau_i \rangle=f(h_i)(\tau_i\pi)$, it suffices to show that $\sum_i||\tau_i||/||\tau||$ is bounded (and independent of the choice of $g$, $\{X_i\}$ and $n:=\#c(g, X)$). This is done by the following lemma.

\begin{lem}\label{lem: norm estimate for state restriction}
Let $\tau\in N_0(G, X)^*$ and $X=\sqcup_iX_i$ be a finite clopen partition. Then $\sum_i||\tau_i||\leq ||\tau||$.
\end{lem}
\begin{proof}
For any $\epsilon>0$, take $a_i\in N_0(G, X)$ such that $||a_i||=1$ and $||\tau_i||\leq |\tau_i(a_i)|+\epsilon/n$ for all $i$.

Then $\sum_i||\tau_i||\leq \sum_i|\tau_i(a_i)|+\epsilon=\sum_i\tau(a_i|_{X_i}\cdot \lambda_i)+\epsilon=\tau(\sum_i(a_i|_{X_i}\cdot\lambda_i))+\epsilon$, where $\lambda_i\in \{\pm 1\}$ is the sign of $\tau_i(a_i)$.

Observe that $\sum_i(a_i|_{X_i}\cdot\lambda_i)\in N_0(G, X)$ and $||\sum_i(a_i|_{X_i}\cdot\lambda_i)||\leq 1$. Hence, $\sum_i||\tau_i||\leq ||\tau||+\epsilon$ for all $\epsilon>0$. So, $\sum_i||\tau_i||\leq ||\tau||$.
\end{proof}

\textbf{Step 2: check $\partial^0S^0=S^1\partial^0$.}\\

Take any $\xi\in N_0(H, Y)^{**}$, $g\in G$ and $\tau\in N_0(G, X)^*$. We do computation as follows.

\begin{align*}
(S^1\partial^0\xi)(g)(\tau)&=\sum_i\langle \pi^{**}((\partial^0\xi)(h_i)), \tau_i \rangle\qquad(\mbox{def. of $S^1$})\\
&=\sum_i\langle (\partial^0\xi)(h_i), \tau_i\pi \rangle\qquad(\mbox{def. of $\pi^{**}$})\\
&=\sum_i\langle h_i\xi, \tau_i\pi \rangle\qquad(\mbox{def. of $\partial^0$})\\
&=\sum_i\langle \xi, h_i^{-1}(\tau_i\pi) \rangle.\qquad(\mbox{def. of dual action})\\
&\\
(\partial^0S^0\xi)(g)(\tau)&=[g(S^0\xi)](\tau)\qquad(\mbox{def. of $\partial^0$})\\
&=(S^0\xi)(g^{-1}\tau)\qquad(\mbox{def. of dual action})\\
&=\langle \pi^{**}\xi , g^{-1}\tau\rangle\qquad(\mbox{$S^0=\pi^{**}$})\\
&=\langle \xi, (g^{-1}\tau)\pi \rangle\qquad(\mbox{def. of $\pi^{**}$})\\
&=\sum_i\langle \xi, (g^{-1}\tau_i)\pi \rangle. \qquad(\mbox{$\tau=\sum_i\tau_i$})
\end{align*}
Therefore, to show the above two equations are equal, it suffices to check that for each $i$, we have $h_i^{-1}(\tau_i\pi)=(g^{-1}\tau_i)\pi$. Equivalently, we need to check that $\tau(\pi(h_i\xi)|_{X_i})=\tau([g(\pi(\xi))]|_{X_i})$ for each $\xi\in N_0(H, Y)$, which is clear by Lemma \ref{lem: module structure under restriction}.

Clearly, $S^0$ is a bijection by Lemma \ref{lemma: key lemma}.
\subsection{Case $n=2$}

Now, we construct a (linear) map $S^2: C_b(H^2, N_0(H, Y)^{**})\to C_b(G^2, N_0(G, X)^{**})$.

Write 
\begin{align}\label{def of X_ijksl for thm 1.2 n=0,1,2}
\begin{split}
X_i&=\{x: c(g_0^{-1}, x)=h_i^{-1}\},\\
X_j&=\{x: c(g_1^{-1}, x)=h_j^{-1}\},\\
X_k&=\{x: c((g_0g_1)^{-1}, x)=h_k^{-1}\},\\
X_s&=\{x: c((g_1g_2)^{-1}, x)=h_s^{-1}\},\\
X_l&=\{x: c(g_2^{-1}, x)=h_l^{-1}\}.
\end{split}
\end{align}

Given $f\in C_b(H^2, N_0(H, Y)^{**})$, we define $S^2(f)=f'$, where for each $\tau\in N_0(G, X)^*$,
\[f'(g_0, g_1)(\tau):=\sum_{i, j}\langle f(h_i, h_j), \tau|_{g_0X_j\cap X_i}\pi\rangle.\]

\textbf{Step 1: $S^2$ is well-defined.}\\

To show $\sup_{(g_0, g_1)\in G^2}||f'(g_0, g_1)||<\infty$, apply Lemma \ref{lem: norm estimate for state restriction} to $X=\sqcup_{i, j}g_0X_j\cap X_i$.\\

\textbf{Step 2: check $\partial^1S^1=S^2\partial^1$.}\\

Take any $f\in C_b(H, N_0(H, Y)^{**})$, write $f'=S^1f$.
For any $g_1, g_2\in G$ and $\tau\in N_0(G, X)^*$, we have

\begin{align*}
&\quad(S^2\partial^1f)(g_1, g_2)(\tau)\\
&=\sum_{j,l}\langle (\partial^1f)(h_j, h_l), \tau|_{g_1X_l\cap X_j}\pi \rangle\qquad(\mbox{def. of $S^2$})\\
&=\sum_{j,l}\langle h_jf(h_l)-f(h_jh_l)+f(h_j), \tau|_{g_1X_l\cap X_j}\pi \rangle\qquad(\mbox{def. of $\partial^1$})\\
&=\sum_{j,l}\langle h_jf(h_l),\tau|_{g_1X_l\cap X_j}\pi \rangle-\sum_{j,l}\langle f(h_jh_l),\tau|_{g_1X_l\cap X_j} \pi\rangle+\sum_{j, l}\langle f(h_j),\tau|_{g_1X_l\cap X_j}\pi \rangle.\\
&\\
&\quad(\partial^1S^1f)(g_1, g_2)(\tau)\\
&=(g_1f'(g_2)-f'(g_1g_2)+f'(g_1))(\tau)\qquad(\mbox{def. of $\partial^1$})\\
&=\sum_l\langle \pi^{**}(f(h_l)),(g_1^{-1}\tau)_l \rangle-\sum_s\langle \pi^{**}(f(h_s)), \tau_s \rangle+\sum_j\langle \pi^{**}(f(h_j)),\tau_j \rangle.\\
\end{align*}
By comparing the two expressions above, it suffices to show each corresponding terms are equal, i.e. 
\begin{align}
\sum_{j,l}\langle h_jf(h_l),\tau|_{g_1X_l\cap X_j}\pi \rangle&=\sum_l\langle \pi^{**}(f(h_l)),(g_1^{-1}\tau)_l \rangle,\label{eq 1 for n=1, thm 1.2}\\
\sum_{j,l}\langle f(h_jh_l),\tau|_{g_1X_l\cap X_j}\pi \rangle&=\sum_s\langle \pi^{**}(f(h_s)), \tau_s \rangle,\label{eq 2 for n=1, thm 1.2}\\
\sum_{j, l}\langle f(h_j),\tau|_{g_1X_l\cap X_j}\pi \rangle&=\sum_j\langle \pi^{**}(f(h_j)),\tau_j \rangle.\label{eq 3 for n=1, thm 1.2}
\end{align}

To check (\ref{eq 1 for n=1, thm 1.2}), it suffices to show for each $l$, 
\[\sum_j\langle h_jf(h_l),\tau|_{g_1X_l\cap X_j}\pi \rangle=\langle \pi^{**}(f(h_l)),(g_1^{-1}\tau)_l \rangle.\] Equivalently, we check $\sum_jh_j^{-1}(\tau|_{g_1X_l\cap X_j}\pi)=(g_1^{-1}\tau)_l\pi$ holds.

Take any $\xi\in N_0(H, Y)$, observe that
\begin{align*}
\sum_jh_j^{-1}(\tau|_{g_1X_l\cap X_j}\pi)(\xi)&=\sum_j(\tau|_{g_1X_l\cap X_j}\pi)(h_j\xi)\qquad(\mbox{def. of dual action})\\
&=\sum_j\tau(\pi(h_j\xi)|_{g_1X_l\cap X_j})\\
&=\sum_j\tau((g_1\pi(\xi))|_{g_1X_l\cap X_j}))\qquad(\mbox{Lemma \ref{lem: module structure under restriction}})\\
&=\tau((g_1\pi(\xi))|_{\sqcup_jg_1X_l\cap X_j}))\\
&=\tau((g_1\pi(\xi))|_{g_1X_l})\qquad(\mbox{$X=\sqcup_jX_j$})\\
&=\tau(g_1(\pi(\xi)|_{X_l})) \\
&=(g_1^{-1}\tau)_l\pi(\xi).
\end{align*}
Hence, $\sum_jh_j^{-1}(\tau|_{g_1X_l\cap X_j}\pi)=(g_1^{-1}\tau)_l\pi$ holds.

To check (\ref{eq 2 for n=1, thm 1.2}), observe that by the cocycle identity, $X_s=\sqcup_{(j, l)\in I_s}(g_1X_l\cap X_j)$ holds, where $I_s:=\{(j, l): h_s=h_jh_l\}$.
\begin{align*}
&\quad\mbox{LHS of (\ref{eq 2 for n=1, thm 1.2})}\\
&=\sum_s\sum_{(j, l)\in I_s}\langle f(h_jh_l),\tau|_{g_1X_l\cap X_j} \pi\rangle~(\mbox{as $g_1X_l\cap X_j\neq \emptyset$ only if $(j, l)\in I_s$ for some $s$})\\
&=\sum_s\sum_{(j, l)\in I_s}\langle f(h_s), \tau|_{g_1X_l\cap X_j}\pi \rangle\\
&=\sum_s\langle f(h_s),\tau|_{\sqcup_{(j, l)\in I_s}g_1X_l\cap X_j} \pi\rangle\\
&=\sum_s\langle f(h_s), \tau|_{X_s} \rangle~(\mbox{as $X_s=\sqcup_{(j,l)\in I_s}(g_1X_l\cap X_j)$})\\
&=\sum_s\langle f(h_s), \tau_s\pi \rangle=\mbox{RHS of (\ref{eq 2 for n=1, thm 1.2})}.
\end{align*}

To check (\ref{eq 3 for n=1, thm 1.2}), it suffices to show for each $j$, $\sum_l\langle f(h_j),\tau|_{g_1X_l\cap X_j}\pi \rangle=\langle \pi^{**}(f(h_j)), \tau_j\rangle$. Observe that
\begin{align*}
\sum_l\langle f(h_j),\tau|_{g_1X_l\cap X_j}\pi \rangle=\langle f(h_j),\tau|_{\sqcup_l g_1X_l\cap X_j}\pi \rangle=\langle f(h_j), \tau|_{X_j}\pi \rangle=\langle \pi^{**}(f(h_j)), \tau_j\rangle.
\end{align*}

\textbf{Step 3: check that $S^1$ is a bijection.} \\

For this purpose, we define a map $T^1: C_b(G, N_0(G, X)^{**})\to C_b(H, N_0(H, Y)^{**})$ by setting $T^1(z)=z'$, where $z'(h)(\nu):=\sum_{i}\langle L^{**}(z(g_i)), \nu_i \rangle$. Here $\nu\in  N_0(H, Y)^*$, $\nu_i=\nu|_{Y_i}, Y_i=\{y: c'(h^{-1}, y)=g_i^{-1}\}$.

Now, we check $T^1S^1=id$ holds. $S^1T^1=id$ can be checked similarly by symmetry.

Take any $f\in C_b(H, N_0(H, Y)^{**})$, $h\in H$ and $\nu\in N_0(H, Y)^*$, write $f'=S^1f$.
\begin{align*}
(T^1S^1f)(h)(\nu)&=\sum_i\langle L^{**}(f'(g_i)), \nu_i\rangle  =\sum_i\langle f'(g_i), \nu_iL \rangle\\
&=\sum_i\sum_j \langle \pi^{**}(f(h_j)), (\nu_iL)_j \rangle\\
&\quad\mbox{(Here, $(\nu_iL)_j:=(\nu_iL)|_{X_{g_i, h_j}}$, $X_{g_i, h_j}=\{x: c(g_i^{-1}, x)=h_j^{-1}\}.$)}\\
&=\sum_i\sum_j \langle f(h_j), (\nu_iL)_j\pi\rangle.
\end{align*}
Now, we claim that $(\nu_iL)_j\pi=\nu|_{\phi(X_{g_i, h_j})\cap Y_i}$.
To see this, take any $\xi\in N_0(H, Y)$, then 
\begin{align*}
(\nu_iL)_j\pi(\xi)&=(\nu_iL)(\pi(\xi)|_{X_{g_i, h_j}})=\nu([L(\pi(\xi)|_{X_{g_i, h_j}})]|_{Y_i})\\
&=\nu(\xi|_{\phi(X_{g_i, h_j})\cap Y_i})\qquad(\mbox{Lemma \ref{lem: composition of pi and L under restrictions}})\\
&=(\nu|_{\phi(X_{g_i, h_j})\cap Y_i})(\xi). 
\end{align*}

Note that $\phi(X_{g_i, h_j})\cap Y_i=Y_i$ if $h_j=h$ and $\emptyset$ otherwise. Therefore,
\[(T^1S^1f)(h)(\nu)=\sum_i\sum_j \langle f(h_j), (\nu_iL)_j\pi\rangle=\sum_i\langle f(h), \nu_i\rangle=\langle f(h), \sum_i \nu_i\rangle=f(h)(\nu).\] So, $T^1S^1=id$. By symmetry, $S^1T^1=id$.

\subsection{Case $n=3$}

Now, we construct a (linear) map $S^3: C_b(H^3, N_0(H, Y)^{**})\to C_b(G^3, N_0(G, X)^{**})$. 
For each $f\in C_b(H^3, N_0(H, Y)^{**})$, define $S^3(f)=f'$, where
\begin{align*}
f'(g_0, g_1, g_2)(\tau):=\sum_{i, j, l}\langle f(h_i, h_j, h_l),\tau|_{(g_0g_1)X_l\cap g_0X_j\cap X_i}\pi \rangle.
\end{align*}
Here, $\tau\in N_0(H, Y)^*$ and we still use (\ref{def of X_ijksl for thm 1.2 n=0,1,2}) for the definition of $X_i, X_j, X_k, X_l, X_s$.\\

\textbf{Step 1: check $S^3$ is well-defined. }\\

The proof is similar to the one when showing $S^1$ is well-defined.\\

\textbf{Step 2: check $S^3\partial^2=\partial^2S^2$.}\\

Take any $f\in C_b(H^3, N_0(H, Y)^{**})$, write $f'=S^2f$, we have
\begin{align*}
&\quad(S^3\partial^2f)(g_0, g_1, g_2)(\tau)\\
&=\sum_{i, j, l}\langle (\partial^2f)(h_i, h_j, h_l), \tau|_{(g_0g_1)X_l\cap g_0X_j\cap X_i}\pi\rangle\\
&=\sum_{i, j, l}\langle h_if(h_j, h_l)-f(h_ih_j, h_l)+f(h_i,h_jh_l)-f(h_i, h_j),\tau|_{(g_0g_1)X_l\cap g_0X_j\cap X_i}\pi \rangle\\
&=\sum_{i, j, l}\langle h_if(h_j, h_l),\tau|_{(g_0g_1)X_l\cap g_0X_j\cap X_i}\pi \rangle-\sum_{i, j, l}\langle f(h_ih_j, h_l),\tau|_{(g_0g_1)X_l\cap g_0X_j\cap X_i}\pi \rangle\\
&+\sum_{i, j, l}\langle f(h_i,h_jh_l),\tau|_{(g_0g_1)X_l\cap g_0X_j\cap X_i}\pi \rangle-\sum_{i, j, l}\langle f(h_i, h_j),\tau|_{(g_0g_1)X_l\cap g_0X_j\cap X_i}\pi \rangle\\
&=:\textcircled{1}-\textcircled{2}+\textcircled{3}-\textcircled{4}.
\end{align*}
Here, $\textcircled{1}, \textcircled{2}, \textcircled{3}$ and $\textcircled{4}$ denote the corresponding 
terms appeared above.

\begin{align*}
&\quad(\partial^2S^2f)(g_0, g_1, g_2)(\tau)\\
&=(g_0f'(g_1, g_2)-f'(g_0g_1, g_2)+f'(g_0, g_1g_2)-f'(g_0, g_1))(\tau)\\
&=\sum_{j, l}\langle f(h_j, h_l), (g_0^{-1}\tau)|_{g_1X_l\cap X_j} \pi\rangle-\sum_{k, l}\langle f(h_k, h_l),\tau|_{g_0g_1X_l\cap X_k} \pi\rangle\\
&+\sum_{i, s}\langle f(h_i, h_s), \tau|_{g_0X_s\cap X_i}\pi \rangle-\sum_{i, j}\langle f(h_i, h_j),\tau|_{g_0X_j\cap X_i} \pi\rangle\\
&=:\textcircled{5}-\textcircled{6}+\textcircled{7}-\textcircled{8}.
\end{align*}

Now, in order to prove $\textcircled{1}-\textcircled{2}+\textcircled{3}-\textcircled{4}=\textcircled{5}-\textcircled{6}+\textcircled{7}-\textcircled{8}$, it suffices to check the corresponding terms are equal, i.e. $\textcircled{1}=\textcircled{5}$, $\textcircled{2}=\textcircled{6}$, $\textcircled{3}=\textcircled{7}$ and $\textcircled{4}=\textcircled{8}$. More precisely, we want to prove the following equalities hold.

\begin{align}
\sum_{i, j, l}\langle h_if(h_j, h_l),\tau|_{(g_0g_1)X_l\cap g_0X_j\cap X_i}\pi \rangle&=\sum_{j, l}\langle f(h_j, h_l), (g_0^{-1}\tau)|_{g_1X_l\cap X_j} \pi\rangle\label{eq 1= eq 5 in thm 1.2 n=2},\\
\sum_{i, j, l}\langle f(h_ih_j, h_l),\tau|_{(g_0g_1)X_l\cap g_0X_j\cap X_i}\pi \rangle&=\sum_{k, l}\langle f(h_k, h_l),\tau|_{g_0g_1X_l\cap X_k} \pi\rangle\label{eq 2= eq 6 in thm 1.2 n=2},\\
\sum_{i, j, l}\langle f(h_i,h_jh_l),\tau|_{(g_0g_1)X_l\cap g_0X_j\cap X_i}\pi \rangle&=\sum_{i, s}\langle f(h_i, h_s), \tau|_{g_0X_s\cap X_i}\pi \rangle\label{eq 3= eq 7 in thm 1.2 n=2},\\
\sum_{i, j, l}\langle f(h_i, h_j),\tau|_{(g_0g_1)X_l\cap g_0X_j\cap X_i}\pi \rangle&=\sum_{i, j}\langle f(h_i, h_j),\tau|_{g_0X_j\cap X_i} \pi\rangle.\label{eq 4= eq 8 in thm 1.2 n=2}
\end{align}

To check (\ref{eq 1= eq 5 in thm 1.2 n=2}), it suffices to check for each $j, l$, 
\begin{align}\label{eq9}
\sum_i\langle h_if(h_j, h_l),\tau|_{(g_0g_1)X_l\cap g_0X_j\cap X_i}\pi \rangle=\langle f(h_j, h_l),(g_0^{-1}\tau)|_{g_1X_l\cap X_j}\pi \rangle.\end{align}
Observe that 
\begin{align}\label{middle eq for thm 1.2, n=2}
h_i^{-1}(\tau|_{(g_0g_1)X_l\cap g_0X_j\cap X_i}\pi)=(g_0^{-1}\tau)|_{g_1X_l\cap X_j\cap g_0^{-1}X_i}\pi.
\end{align}
Indeed, take any $\xi\in N_0(H, Y)$, and write $X_0:=(g_0g_1)X_l\cap g_0X_j\cap X_i$, we have 
$(h_i^{-1}(\tau|_{X_0}\pi))\xi=\tau(\pi(h_i\xi)|_{X_0})\overset{Lem. \ref{lem: module structure under restriction}}{=}\tau(g_0(\pi(\xi)|_{g_0^{-1}X_0}))=[(g_0^{-1}\tau)|_{g_0^{-1}X_0}\pi](\xi)$.

Therefore,
\begin{align*}
\mbox{LHS of (\ref{eq9})}
&=\sum_i\langle h_if(h_j, h_l),\tau|_{(g_0g_1)X_l\cap g_0X_j\cap X_i}\pi \rangle\\
&=\sum_i\langle f(h_j, h_l),h_i^{-1}(\tau|_{(g_0g_1)X_l\cap g_0X_j\cap X_i}\pi) \rangle\\
&\overset{(\ref{middle eq for thm 1.2, n=2})}{=}\sum_i\langle f(h_j, h_l),(g_0^{-1}\tau)|_{g_1X_l\cap X_j\cap g_0^{-1}X_i}\pi) \rangle\\
&=\langle f(h_j, h_l),(g_0^{-1}\tau)|_{\sqcup_i(g_1X_l\cap X_j\cap g_0^{-1}X_i)}\pi \rangle\\
&=\langle f(h_j, h_l),(g_0^{-1}\tau)|_{g_1X_l\cap X_j} \pi\rangle\\
&=\mbox{RHS of (\ref{eq9})}. 
\end{align*}

To check (\ref{eq 2= eq 6 in thm 1.2 n=2}),
observe $X_k=\sqcup_{(i, j)\in I_k}g_0X_j\cap X_i$, where $I_k=\{(i, j): h_k=h_ih_j\}$.
So,
\begin{align*}
\mbox{RHS of (\ref{eq 2= eq 6 in thm 1.2 n=2})}
&= \sum_{k,l}\sum_{(i, j)\in I_k} \langle f(h_ih_j, h_l), \tau|_{g_0g_1X_l\cap g_0X_j\cap X_i}\pi \rangle\\
&=\sum_l\sum_k\sum_{(i, j)\in I_k} \langle f(h_ih_j, h_l), \tau|_{g_0g_1X_l\cap g_0X_j\cap X_i}\pi \rangle\\
&=\sum_{l, i,j}\langle f(h_ih_j, h_l), \tau|_{g_0g_1X_l\cap g_0X_j\cap X_i}\pi \rangle\\
&=\mbox{LHS of (\ref{eq 2= eq 6 in thm 1.2 n=2})}.
\end{align*}
Here, the 2nd last equality holds as $g_0g_1X_l\cap g_0X_j\cap X_i\neq \emptyset$ only if $(i, j)\in I_k$ for some $k$.

To check (\ref{eq 3= eq 7 in thm 1.2 n=2}), use $X_s=\sqcup_{(j, l)\in I_s}g_1X_l\cap X_j$, where $I_s=\{(j, l): h_s=h_jh_l\}$ to deduce
\begin{align*}
\mbox{RHS of (\ref{eq 3= eq 7 in thm 1.2 n=2})}
&=\sum_{i, s}\langle f(h_i, h_s),\tau|_{g_0X_s\cap X_i}\pi \rangle \\
&=\sum_{i, s}\sum_{(j, l)\in I_s}\langle f(h_i, h_jh_l),\tau|_{g_0g_1X_l\cap g_0X_j\cap X_i}\pi \rangle\\
&= \sum_{k,l}\sum_{(i, j)\in I_k}\langle f(h_i, h_jh_l),\tau|_{g_0g_1X_l\cap g_0X_j\cap X_i}\pi \rangle\\
&=\sum_{l,i,j}\langle f(h_i, h_jh_l),\tau|_{g_0g_1X_l\cap g_0X_j\cap X_i}\pi \rangle\\
&=\mbox{LHS of (\ref{eq 3= eq 7 in thm 1.2 n=2})}.
\end{align*}
The 3rd last equality holds since we have a bijection between the index sets:

\[\{(i, s, j, l): h_s=h_jh_l, g_0g_1X_l\cap g_0X_j\cap X_i\neq \emptyset\}\]
and
\[\{(k, l, i, j): h_k=h_ih_j, g_0g_1X_l\cap g_0X_j\cap X_i\neq \emptyset\}.\]

And the 2nd last equality holds as $g_0g_1X_l\cap g_0X_j\cap X_i\neq \emptyset$ only if $(i, j)\in I_k$ for some $k$.

To check (\ref{eq 4= eq 8 in thm 1.2 n=2}), just observe $g_0X_j\cap X_i=\sqcup_l(g_0g_1X_l\cap g_0X_j\cap X_i).$\\

\textbf{Step 3: $S^2$ is a bijection.}\\

Define $T^2: C_b(G^2, N_0(G, X)^{**})\to C_b(H^2, N_0(H, Y)^{**})$ by setting $T^2(f')=f''$, where 
\[f''(h_0, h_1)(\nu)=\sum_{i, j}\langle f'(g_i, g_j),\nu|_{h_0Y_j\cap Y_i} L\rangle.\] Here $\nu\in N_0(H, Y)^*$, and 
\begin{align*}
Y_i&:=\{y\in Y; c'(h_0^{-1}, y)=g_i^{-1}\},\\
Y_j&:=\{y\in Y; c'(h_1^{-1}, y)=g_j^{-1}\}. 
\end{align*}

Now, let us check $T^2S^2=id$.

Take any $f\in C_b(H^2, N_0(H, Y)^{**})$, $\nu\in N_0(H, Y)^*$, let $f'=S^2f$ and $f''=T^2f'$, we aim to show $f''(h_0, h_1)(\nu)=f(h_0, h_1)(\nu)$.

A calculation shows 
\begin{align*}
f''(h_0, h_1)(\nu)&=\sum_{i, j}\sum_{s, t}\langle f(h_s, h_t), (\nu|_{h_0Y_j\cap Y_i}L)|_{g_iX_{g_j, h_t}\cap X_{g_i, h_s}}\pi \rangle\\
&=\sum_{i, j}\sum_{s, t}\langle f(h_s, h_t), \nu|_{h_0Y_j\cap Y_i\cap \phi(g_iX_{g_j, h_t}\cap X_{g_i, h_s})}\rangle.
\end{align*}
Here, $X_{g_i, h_s}=\{x\in X: c(g_i^{-1}, x)=h_s^{-1}\}$ and $X_{g_j, h_t}=\{x\in X: c(g_j^{-1}, x)=h_t^{-1}\}$.

One can check that $h_0Y_j\cap Y_i\cap \phi(g_iX_{g_j, h_t}\cap X_{g_i, h_s})=\emptyset$ unless $(h_s, h_t)=(h_0, h_1)$. If this condition holds, then $h_0Y_j\cap Y_i\cap \phi(g_iX_{g_j, h_1}\cap X_{g_i, h_0})=h_0Y_j\cap Y_i$.

Hence, 
\begin{align*}
f''(h_0, h_1)(\nu)=\sum_{i,j}\langle f(h_0, h_1), \nu|_{h_0Y_j\cap Y_i}\rangle=f(h_0, h_1)(\nu).
\end{align*}

\subsection{General case}

We'll define a map $S^n: C_b(H^n, N_0(H, Y)^{**})\to C_b(G^n, N_0(G, X)^{**})$ and check it is a bijective cochain map. Hence, it induces an isomorphism between the two bounded cohomology groups.

We introduce some notations which will be used in this subsection.

Let $g_0,\ldots, g_n$ be elements in $G$. Write 
\begin{align}\label{def for thm 1.2: c(g, X)}
\begin{split}
c(g_j^{-1}, X)&:=\{h_{t_j}^{-1}|~ t_j\},\\
X_{t_j}&:=\{x\in X: c(g_j^{-1}, x)=h_{t_j}^{-1}\},\\
c((g_ig_{i+1})^{-1}, X)&:=\{h_{s_i}^{-1}|~ s_i\},\\
X_{s_i}&:=\{x\in X: c((g_ig_{i+1})^{-1}, x)=h_{s_i}^{-1}\}.
\end{split}
\end{align}
Note that $X_{s_i}=\sqcup_{(t_i, t_{i+1})\in \Delta_i}X_{t_i}\cap g_iX_{t_{i+1}}$, where $\Delta_i:=\{(t_i, t_{i+1}): h_{s_i}=h_{t_i}h_{t_{i+1}}\}$.

We remind the reader once again that $X_{t_j}$,  $X_{s_i}$ and $\Delta_i$ are just simplified versions for the notations $X_{g_j, h_{t_j}}$, $X_{g_ig_{i+1}, h_{s_i}}$ and $\Delta_{s_i}$ respectively, which we will switch to if necessary.

Let $f\in C_b(H^n, N_0(H, Y)^{**})$, we define $S^n(f):=f'\in C_b(G^n, N_0(G, X)^{**})$, where 
\[f'(g_0, \dots, g_{n-1})(\tau):=\sum_{t_0,\dots, t_{n-1}}\langle f(h_{t_0}, \dots, h_{t_{n-1}}), \tau|_{[t_0,\ldots, t_{n-1}]}\circ\pi \rangle\]
for all $g_i\in G$ and $\tau\in N_0(G, X)^*$. 

Here, 
\begin{align}\label{def of t_0,...,t_n-1 in general case in thm 1.2}
[t_0, \ldots, t_{n-1}]:=X_{t_0}\cap g_0X_{t_1}\cap g_0g_1X_{t_2}\cap \dots\cap (g_0\cdots g_{n-2})X_{t_{n-1}}.
\end{align}
We split the proof into several steps.\\

\textbf{Step 1: $S^n$ is well-defined.}\\

Apply Lemma \ref{lem: norm estimate for state restriction} to $X=\sqcup_{t_0,\ldots, t_{n-1}}[t_0, \ldots, t_{n-1}]$.\\

\textbf{Step 2: $S^n$ is a cochain map, i.e. $S^{n+1}\partial^n=\partial^nS^n$.}\\

Fix $f\in C_b(H^n, N_0(H, Y)^{**})$, let $f'=S^nf$. We aim to show $S^{n+1}\partial^nf=\partial^nf'$. 

Take any $g_0,\ldots, g_n$ in $G$, and any $\tau\in N_0(G, X)^*$. We compute as follows.

First,
\begin{align}\label{eq for thm 1.2: RHS of cochain map}
\begin{split}
&\qquad\partial^nf'(g_0,\ldots, g_n)(\tau)\\
&=[g_0f'(g_1,\ldots, g_n)+\sum_{i=1}^n(-1)^if'(g_0, \ldots, g_{i-1}g_i, \ldots, g_n)+(-1)^{n+1}f'(g_0,\ldots, g_{n-1})](\tau)\\
&=\sum_{t_1,\ldots, t_n}\langle f(h_{t_1},\ldots, h_{t_n}), (g_0^{-1}\tau)|_{[t_1,\ldots, t_n]}\pi \rangle\\
&\quad+\sum_{i=1}^n(-1)^i\sum_{\substack{t_0,\ldots, t_{i-2}, \\s_{i-1}, t_{i+1},\ldots, t_n}}\langle f(h_{t_0},\ldots, h_{t_{i-2}}, h_{s_{i-1}}, h_{t_{i+1}},\ldots,h_{t_n}), \tau|_{[t_0,\ldots, t_{i-2}, s_{i-1}, t_{i+1},\ldots, t_n]}\pi \rangle\\
&\quad+(-1)^{n+1}\sum_{t_0,\ldots, t_{n-1}}\langle f(h_{t_0},\ldots, h_{t_{n-1}}), \tau|_{[t_0,\ldots, t_{n-1}]}\pi \rangle.   
\end{split}
\end{align}
Recall here, $c(g_{i-1}g_i, X)=\{h_{s_{i-1}}|~s_{i-1}\}$, $X_{s_{i-1}}=\{x: c((g_{i-1}g_i)^{-1}, x)=h_{s_{i-1}}^{-1}\}$ and
\begin{multline}\label{def of t_0...t_i-2,s_i-1,t_i+1,...,t_n}
[t_0,\ldots, t_{i-2}, s_{i-1}, t_{i+1},\ldots, t_n]=X_{t_0}\cap g_0X_{t_1}\cap \cdots\cap (g_0\cdots g_{i-2})X_{s_{i-1}}\\
\cap(g_0\cdots g_{i-2}g_{i-1}g_i)X_{t_{i+1}}\cap\cdots\cap (g_0\dots g_{n-1})X_{t_n}.
\end{multline}
Second, 
\begin{align}\label{eq for thm 1.2: LHS of cochain map}
\begin{split}
&\qquad(S^{n+1}\partial^nf)(g_0,\ldots, g_n)(\tau)\\
&=\sum_{t_0,\ldots, t_n}\langle\partial^n(f)(h_{t_0},\ldots, h_{t_n}), \tau|_{[t_0,\ldots, t_n]}\pi\rangle\\
&=\sum_{t_0,\ldots, t_n}\langle h_{t_0}f(h_{t_1},\ldots, h_{t_n}),  \tau|_{[t_0,\ldots, t_n]} \pi\rangle\\
&\quad+\sum_{t_0,\ldots, t_n}\langle \sum_{i=1}^n(-1)^i f(h_{t_0},\ldots, h_{t_{i-1}}h_{t_i},\ldots, h_{t_n}), \tau|_{[t_0,\ldots, t_n]}\pi\rangle\\
&\quad+\sum_{t_0,\ldots, t_n}\langle(-1)^{n+1}f(h_{t_0},\ldots, h_{t_{n-1}}), \tau|_{[t_0,\ldots, t_n]} \pi\rangle.
\end{split}
\end{align}

By comparing (\ref{eq for thm 1.2: LHS of cochain map}) with (\ref{eq for thm 1.2: RHS of cochain map}), it suffices to show that the corresponding terms are equal, i.e. we want to establish the following identities.

(a) \begin{align*}
\sum_{t_1,\ldots, t_n}\langle f(h_{t_1},\ldots, h_{t_n}), (g_0^{-1}\tau)|_{[t_1,\ldots, t_n]}\pi \rangle=\sum_{t_0,\ldots, t_n}\langle h_{t_0}f(h_{t_1},\ldots, h_{t_n}),  \tau|_{[t_0,\ldots, t_n]} \pi\rangle.
\end{align*}

(b) For each $1\leq i\leq n$,
\begin{multline*}
\sum_{\substack{t_0,\ldots, t_{i-2}, \\s_{i-1}, t_{i+1},\ldots, t_n}}\langle f(h_{t_0},\ldots, h_{t_{i-2}}, h_{s_{i-1}}, h_{t_{i+1}},\ldots,h_{t_n}), \tau|_{[t_0,\ldots, t_{i-2}, s_{i-1}, t_{i+1},\ldots, t_n]}\pi \rangle\\
=\sum_{t_0,\ldots, t_n}\langle  f(h_{t_0},\ldots, h_{t_{i-1}}h_{t_i},\ldots, h_{t_n}), \tau|_{[t_0,\ldots, t_n]}\pi\rangle.
\end{multline*}

(c) \begin{align*}
\sum_{t_0,\ldots, t_{n-1}}\langle f(h_{t_0},\ldots, h_{t_{n-1}}), \tau|_{[t_0,\ldots, t_{n-1}]}\pi \rangle=\sum_{t_0,\ldots, t_n}\langle (-1)^{n+1}f(h_{t_0},\ldots, h_{t_{n-1}}), \tau|_{[t_0,\ldots, t_n]} \pi\rangle.
\end{align*}

\begin{proof}[Proof of (a)]
By Lemma \ref{lem: composition of pi and L under restrictions}, $(g_0^{-1}\tau)|_{g_0^{-1}X_{t_0}\cap [t_1,\ldots, t_n]}\pi=h_{t_0}^{-1}(\tau|_{[t_0,\ldots, t_n]}\pi)$ holds. 

Indeed, evaluate both sides at any $\xi$ in $N_0(H, Y)$, this boils down to check $X_{t_0}\cap g_0[t_1,\ldots,t_n]=[t_0,\ldots,t_n]$, which is clear from the definition of $[t_0,\ldots,t_n].$

Using this identity, we deduce
\begin{align*}
&\quad \mbox{RHS of (a)}\\
&=\sum_{t_0,\ldots, t_n}\langle f(h_{t_1},\ldots, h_{t_n}),h_{t_0}^{-1}(\tau|_{[t_0,\ldots,t_n]}\pi) \rangle\\
&=\sum_{t_0,\ldots, t_n}\langle f(h_{t_1},\ldots, h_{t_n}), (g_0^{-1}\tau)|_{g_0^{-1}X_{t_0}\cap [t_1,\ldots, t_n]}\pi\rangle\\
&=\sum_{t_1,\ldots, t_n}\langle f(h_{t_1},\ldots, h_{t_n}), (g_0^{-1}\tau)|_{\sqcup_{t_0}g_0^{-1}X_{t_0}\cap [t_1,\ldots, t_n]}\pi\rangle\\
&=\sum_{t_1,\ldots, t_n}\langle f(h_{t_1},\ldots, h_{t_n}), (g_0^{-1}\tau)|_{[t_1,\ldots, t_n]}\pi\rangle\\
&=\mbox{LHS of (a)}.\qedhere
\end{align*}
\end{proof}

\begin{proof}[Proof of (b)]
From $X_{s_{i-1}}=\sqcup_{(t_{i-1}, t_i)\in \Delta_{i-1}}X_{t_{i-1}}\cap g_{i-1}X_{t_i}$,
we deduce that
\begin{align}\label{eq10}
[t_0,\ldots, t_{i-2}, s_{i-1}, t_{i+1},\ldots, t_n]=\sqcup_{(t_{i-1}, t_i)\in \Delta_{i-1}}[t_0,\ldots, t_n].
\end{align}

Indeed, just compare definition (\ref{def of t_0,...,t_n-1 in general case in thm 1.2}) and (\ref{def of t_0...t_i-2,s_i-1,t_i+1,...,t_n}).

Hence, 
\begin{align*}
&\quad\sum_{s_{i-1}}\langle f(h_{t_0},\ldots, h_{t_{i-2}}, h_{s_{i-1}}, h_{t_{i+1}},\ldots,h_{t_n}), \tau|_{[t_0,\ldots, t_{i-2}, s_{i-1}, t_{i+1},\ldots, t_n]}\pi \rangle\\
&\overset{(\ref{eq10})}{=}\sum_{s_{i-1}}\sum_{(t_{i-1}, t_i)\in \Delta_{i-1}}\langle f(h_{t_0},\ldots, h_{t_{i-2}}, h_{t_{i-1}}h_{t_i}, h_{t_{i+1}},\ldots,h_{t_n}), \tau|_{[t_0,\ldots,t_n]}\pi \rangle\\
&=\sum_{t_{i-1}, t_i}\langle f(h_{t_0},\ldots, h_{t_{i-2}}, h_{t_{i-1}}h_{t_i}, h_{t_{i+1}},\ldots,h_{t_n}), \tau|_{[t_0,\ldots,t_n]}\pi \rangle.
\end{align*}
The last equality holds since $[t_0, \ldots, t_n]\neq \emptyset$ only if $(t_{i-1}, t_i)\in \Delta_{i-1}$ for some $s_{i-1}$. (Recall $\Delta_{i-1}=\Delta_{s_{i-1}}$ for simplicity.)

 Therefore, (b) holds by taking sum over $t_0,\ldots, t_{i-2},t_{i+1},\ldots, t_n$ on both sides of the above equality. 
\end{proof}

\begin{proof}[Proof of (c)]
As $[t_0,\ldots, t_n]=[t_0,\ldots,t_{n-1}]\cap (g_0\cdots g_{n-1})X_{t_n}$, we get 
\begin{align*}
&\quad\mbox{RHS of (c)}\\
&=\sum_{t_0,\ldots, t_n}\langle f(h_{t_0},\ldots, h_{t_{n-1}}), \tau|_{[t_0,\ldots, t_{n-1}]\cap (g_0\cdots g_{n-1})X_{t_n}}\pi \rangle\\
&=\sum_{t_0,\ldots, t_{n-1}}\langle f(h_{t_0},\ldots, h_{t_{n-1}}), \tau|_{[t_0,\ldots, t_{n-1}]\cap \sqcup_{t_n}(g_0\cdots g_{n-1})X_{t_n}}\pi \rangle\\
&=\sum_{t_0,\ldots, t_{n-1}}\langle f(h_{t_0},\ldots, h_{t_{n-1}}), \tau|_{[t_0,\ldots, t_{n-1}]}\pi \rangle\\
&=\mbox{LHS of (c)}.\qedhere
\end{align*}
\end{proof}

\textbf{Step 3: $S^n$ is a bijection.}\\

To check $S^n$ is a bijection, consider the natural inverse map $T^n$ defined by symmetry as follows. 

For $f'\in C_b(G^n, N_0(G, X)^{**})$, define $f'':=T(f')$ as follows:
\begin{align*}
f''(h_0,\ldots, h_{n-1})(\nu):=\sum_{i_0,\ldots, i_{n-1}}\langle f'(g_{i_0},\ldots, g_{i_{n-1}}), \nu|_{[i_0,\ldots, i_{n-1}]} \circ L\rangle.
\end{align*}

Here, $\nu\in N_0(H, Y)^*$, $[i_0, \ldots, i_{n-1}]:=Y_{i_0}\cap h_0Y_{i_1}\cap\cdots\cap (h_0\cdots h_{n-2})Y_{i_{n-1}}$ and $Y_{i_j}:=\{y: c'(h_j^{-1}, y)=g_{i_j}^{-1}\}$.

Let $f\in C_b(H^n, N_0(H, Y)^{**})$, write $f'=S^nf$ and $f''=T^nf'$.
We aim to show $T^nS^n=id$, i.e. $f''=f$.
Take any $h_0,\ldots, h_{n-1}$ in $H$ and any $\nu$ in $N_0(H, Y)^*$, then
\begin{align}\label{eq for thm 1.2: main computation on bijection}
\begin{split}
&\quad f''(h_0,\ldots, h_{n-1})(\nu)\\
&=\sum_{i_0,\ldots, i_{n-1}}\langle f'(g_{i_0},\ldots, g_{i_{n-1}}), \nu|_{[i_0,\ldots, i_{n-1}]} \circ L\rangle\\
&=\sum_{i_0,\ldots, i_{n-1}}\sum_{t_0,\ldots, t_{n-1}}\langle f(h_{t_0}, \ldots, h_{t_{n-1}}),(\nu|_{[i_0,\ldots, i_{n-1}]} \circ L)|_{[t_0,\ldots, t_{n-1}]}\pi \rangle.
\end{split}
\end{align}
Here, $[t_0,\ldots, t_{n-1}]:=X_{g_{i_0}, h_{t_0}}\cap g_{i_0}X_{g_{i_1}, h_{t_1}}\cap\cdots\cap (g_{i_0}\cdots g_{i_{n-2}})X_{g_{i_{n-1}}, h_{t_{n-1}}}$ and $X_{g_{i_j}, h_{t_j}}:=\{x: c(g_{i_j}^{-1}, x)=h_{t_j}^{-1}\}$.

To continue the proof, we notice the following facts.

(Fact 5) $(\nu|_{[i_0,\ldots, i_{n-1}]} \circ L)|_{[t_0,\ldots, t_{n-1}]}\pi=\nu|_{[i_0,\ldots, i_{n-1}]\cap \phi([t_0,\ldots, t_{n-1}])}$.

To prove this fact, take any $\xi\in N_0(H, Y)$, $h\in H$ and $y\in Y$. Let $\xi'=\pi(\xi)$ and $\eta=\xi'|_{[t_0,\ldots,t_{n-1}]}$. We aim to prove $((\nu|_{[i_0,\ldots, i_{n-1}]} \circ L)|_{[t_0,\ldots, t_{n-1}]}\pi)(\xi)=(\nu|_{[i_0,\ldots, i_{n-1}]\cap \phi([t_0,\ldots, t_{n-1}])})(\xi)$; equivalently, $\nu((L\circ \eta)|_{[i_0,\ldots, i_{n-1}]})=\nu(\xi|_{{[i_0,\ldots, i_{n-1}]\cap \phi([t_0,\ldots, t_{n-1}])}})$. This is clear by Lemma \ref{lem: composition of pi and L under restrictions}.

(Fact 6) $[i_0,\ldots, i_{n-1}]\cap \phi([t_0,\ldots, t_{n-1}])=\emptyset$ unless $h_{t_j}=h_j$ for all $0\leq j\leq n-1$. When these conditions hold, the intersection equals $[i_0,\ldots, i_{n-1}]$.

This fact can be directly checked using cocycle identity and Lemma \ref{lemma: xin's lemma}.\\

Now we can continue the computation in (\ref{eq for thm 1.2: main computation on bijection}) as follows.
\begin{align*}
\begin{split}
&\quad f''(h_0,\ldots, h_{n-1})(\nu)\\
&\overset{Fact~5}{=}\sum_{i_0,\ldots, i_{n-1}}\sum_{t_0,\ldots, t_{n-1}}\langle f(h_{t_0}, \ldots, h_{t_{n-1}}),\nu|_{[i_0,\ldots, i_{n-1}]\cap \phi([t_0,\ldots, t_{n-1}])} \rangle\\
&\overset{Fact~6}{=}\sum_{i_0,\ldots, i_{n-1}}\langle f(h_0, \ldots, h_{n-1}),\nu|_{[i_0,\ldots, i_{n-1}]} \rangle\\
&=\langle f(h_0,\ldots,h_{n-1}), \nu \rangle.
\end{split}
\end{align*}
Therefore, $f''=f$ and the proof is finished.

\section{Concluding remarks}\label{section: remarks}

We end this paper with several remarks addressing various issues related to our main theorems.

(1) It is routine to check that the isomorphism $S$ constructed in the proof preserves the Johnson classes \cite[Definition 20]{bnnw2} and the fundamental classes \cite[Definition 3]{bnnw} for the two actions, which were introduced to study amenability of group actions. Since the triviality of the Johnson class \cite[Theorem 1]{bnnw2} and the nontriviality of the fundamental class \cite[Theorem 9]{bnnw} are used to characterize the amenability of an action, we see that for two topologically free actions, topologically amenability is an invariant property under continuous orbit equivalence.

(2) Under the same assumptions as in our theorems, it is natural to ask whether $H^{uf}_n(G\curvearrowright X)\cong H^{uf}_n(H\curvearrowright Y)$ for $n\geq 1$ holds. For this question, our method does not work. Indeed, we use crucially the decomposition of elements in $N_0(G, X)$ as a finite sum of elements in $N_0(G, X)$ with respect to a finite partition of $X$ into clopen subsets. The summands do not belong to $W_0(G, X)$ anymore if we decompose an element in $W_0(G, X)$. This may be interpreted as saying $W_0(G, X)$ is not res-invariant (w.r.t. $X$), following \cite[Definition 4.1]{li} in spirit or reflecting the fact $W_0(G, X)$ is not a $(G, X)$-module in the sense of \cite{monod}.

In fact, the following suggests the above isomorphism may fail in general. First, by \cite[Corollary 5]{bnnw}, we know that for a topologically amenable action $G\curvearrowright X$, $H^{uf}_n(G\curvearrowright X)\cong H_n(G, \mathbb{R})\oplus H_n(G, N_0(G, X)^*)$ holds. Then, we consider two continuous actions which are both topologically amenable and topologically free, e.g. left translation actions on the Stone-\u{C}ech compactifications of free groups with different rank, say $F_2$ and $F_3$. As $F_2$ is quasi-isometric to $F_3$, we know they are also bilipschitz equivalent by \cite{whyte}; equivalently (see \cite[Corollary 2.21]{li}), $F_2\curvearrowright \beta F_2\overset{coe}{\sim} F_3\curvearrowright \beta F_3$. On the one hand, $H_1(F_k,\mathbb{R})\cong \mathbb{R}^k$ for all $k\geq 1$ implies $H_1(F_2,\mathbb{R})\not\cong H_1(F_3,\mathbb{R})$, but on the other hand, $H_1(F_2, N_0(F_2, \beta F_2)^*)\cong H_1(F_3, N_0(F_3, \beta F_3)^*)$ by our main theorems. This suggests that for the above example, one expects these two actions have non-isomorphic $H_1^{uf}$.

%Since $H_0(G, \mathbb{R})\cong \R$, $H^{uf}_0(G\curvearrowright X)\cong H^{uf}_0(H\curvearrowright Y)$ is reasonable in view of \cite[Corollary 10]{bnnw}.
Nevertheless, under certain assumptions, we can still have some positive result.

\begin{cor}
Let $G\curvearrowright X$ and $H\curvearrowright Y$ be  topologically free actions which are COE. If $G\curvearrowright X$ is topologically amenable and both $G$ and $H$ are  finitely generated torsion free nilpotent groups with finite cohomological dimension, then $H_i^{uf}(G\curvearrowright X)\cong H_i^{uf}(H\curvearrowright Y)$ for all $i\geq 0$.
\end{cor}
\begin{proof}
First, observe that $H\curvearrowright Y$ is also topologically amenable. This is clear by previous Remark (1). One can also check this using \cite[Definition 7]{bnnw} and the map $\pi$ defined in Lemma \ref{lemma: key lemma}. 

Then by \cite[Corollary 5]{bnnw}, for topologically amenable actions, we have $H^{uf}_i(G\curvearrowright X)\cong H_i(G, \mathbb{R})\oplus H_i(G, N_0(G, X)^*)$. The same holds for the action $H\curvearrowright Y$.

By Theorem \ref{thm on uf homology}, it suffices to show that $H_i(G, \mathbb{R})\cong H_i(H, \mathbb{R})$.

Observe that COE between the above two actions implies $G$ and $H$ are quasi-isometric. Indeed, just fix any $x\in X$, it is easy to check $g\in G\to H\ni c(g, x)$ is a quasi-isometry using Lemma \ref{lemma: xin's lemma}.

Now, as $G$ and $H$ are quasi-isometric nilpotent groups, we deduce $H^i(G, \mathbb{R})\cong H^i(H, \mathbb{R})$ by \cite[Theorem 1.5]{sauer}. Moreover, $cd(G)=cd(H)$ by \cite[Theorem 1.2]{sauer} or \cite[Corollary 4.42]{li}. The proof is finished by noticing that for finitely generated torsion free nilpotent groups, they are  
orientable Poincar\'e duality groups, which implies that $H^i(G, \mathbb{R})\cong H_{n-i}(G, \mathbb{R})$, where $n=cd(G)$. And the same holds for $H$. (For the above assertion and the definition of orientable Poincar\'e duality groups, see Section 10, Chapter VIII in \cite{brown}, in particular, \cite[Example 1, p. 222]{brown}.)
\end{proof}

(3) (Co)homology groups associated to many coefficient modules are proved to be invariants under COE for two topologically free actions in \cite[Theorem 3.1, 3.5]{li}, but as far as we can see, the (co)homology groups considered in our paper are not covered by these theorems. It seems plausible one may also prove our theorems using the method in \cite{li}, i.e. try to interpret the (co)homologies for $G$ as (co)homologies for the transformation groupoid, but one may need to extend unitary representations of \'etale locally compact groupoid (\cite[\S 3.2]{li}, \cite{renault}) to linear isometric representations on Banach spaces. In fact, for two COE topologically free actions $G\curvearrowright X$ and $H\curvearrowright Y$, it may be possible to show there is a one to one correspondence between $(G, X)$-modules of type $M$ in the sense of \cite{monod}, say $E$, and $(H, Y)$-modules of type $M $, say $F$, and under this correspondence, $H_b^*(G, E^*)\cong H_b^*(H, F^*)$. Theorem \ref{thm on b cohomology} may be thought of as an evidence for this.\\

\textbf{Acknowledgement:} We thank Prof. Piotr Nowak for helpful discussion related to this paper. We are also very grateful to the anonymous referee. He/She provided us with many valuable suggestions which greatly improved the readability of the paper.

\end{document}